\documentclass[pdflatex,sn-mathphys-num]{sn-jnl}

\usepackage{graphicx}
\usepackage{multirow}
\usepackage{amsmath,amssymb,amsfonts}
\usepackage{amsthm}
\usepackage{mathrsfs}
\usepackage[title]{appendix}
\usepackage{xcolor}
\usepackage{textcomp}
\usepackage{manyfoot}
\usepackage{booktabs}
\usepackage{algorithm}
\usepackage{algorithmicx}
\usepackage{algpseudocode}
\usepackage{listings}
\usepackage{tikz-cd}
\usepackage{stmaryrd}
\usepackage{mathdots}
\usepackage{csquotes}
\mathchardef\mhyphen="2D

\newcommand{\lr}{\Leftrightarrow}
\newcommand{\ri}{\Rightarrow}
\newcommand{\mm}{^{-1}}

\newcommand{\Hh}{\mathbb{H}}
\newcommand{\Nn}{\mathbb{N}}
\newcommand{\Qq}{\mathbb{Q}}
\newcommand{\Rr}{\mathbb{R}}

\newcommand{\Zz}{\mathbb{Z}}

\newcommand{\fF}{\mathcal{F}}

\newcommand{\CC}{\mathscr{C}}

\newcommand{\KK}{\mathscr{K}}

\newcommand{\RR}{\mathscr{R}}

\newcommand{\XX}{\mathscr{X}}

\newcommand{\Aym}{\mathfrak{A}}
\newcommand{\Fym}{\mathfrak{F}}
\newcommand{\Gym}{\mathfrak{G}}
\newcommand{\Sym}{\mathfrak{S}}

\newcommand{\cym}{\mathfrak{c}}

\newcommand{\iym}{\mathfrak{i}}
\newcommand{\pym}{\mathfrak{p}}

\newcommand{\adap}[1]{\left( #1 \right)}

\newcommand{\adacc}[1]{\left\llbracket #1 \right\rrbracket}

\newcommand{\Ker}[1]{\text{Ker}\adap{#1}}

\newcommand{\Id}{\text{Id}}
\newcommand{\id}[1]{\textbf{1}_#1}

\newcommand{\Aut}[1]{\text{Aut}\adap{#1}}
\newcommand{\Der}[1]{\text{Der}\adap{#1}}

\newcommand{\End}[1]{\text{End}\adap{#1}}

\newcommand{\cEq}[1]{\text{Eq}\adap{#1}}

\newcommand{\ob}{\textbf{ob}}
\newcommand{\mor}{\textbf{mor}}

\newcommand{\conj}[1]{\text{conj}_{#1}}

\newcommand{\mapp}[5]{\begin{matrix}#1&:&#2&\longrightarrow &#3\\&&#4&\longmapsto &#5\end{matrix}}
\newcommand{\func}[7]{\begin{matrix}#1&:&#2&\longrightarrow &#3\\&&#4\in \ob#2&\longmapsto &#5\\&&#6\in \mor{#2}&\longmapsto &#7\end{matrix}}

\newcommand{\cAlg}[1]{(#1)\mhyphen\textsf{Alg}}

\newcommand{\cCat}{\textsf{Cat}}
\newcommand{\cgrp}{\textsf{Grp}}

\newcommand{\cmon}{\textsf{Mon}}
\newcommand{\cpogrp}{\textsf{OrdGrp}}

\newcommand{\crpogrp}{\textsf{ROrdGrp}}

\newcommand{\cset}{\textsf{Set}}
\newcommand{\cpreord}{\textsf{PreOrd}}

\newcommand{\cLie}[1]{\textsf{Lie}_{#1}}

\newcommand{\cpxsmod}[1]{\textsf{S}\mhyphen\textsf{PreXMod}\adap{#1}}

\newcommand{\cpxmod}[1]{\textsf{PreXMod}\adap{#1}}
\newcommand{\cxsmod}[1]{\textsf{S}\mhyphen\textsf{XMod}\adap{#1}}
\newcommand{\cxsmodp}[1]{\textsf{S}\mhyphen\textsf{XMod}^+\adap{#1}}
\newcommand{\lnorm}[1]{\textsf{Norm}\adap{#1}}

\newcommand{\cSchCat}[1]{\textsf{S}\mhyphen\textsf{Cat}\adap{#1}}
\newcommand{\cSchGrpd}[1]{\textsf{S}\mhyphen\textsf{Grpd}\adap{#1}}

\newcommand{\cPt}[1]{\textsf{SplExt}\adap{#1}}
\newcommand{\cSchRG}[1]{\textsf{S}\mhyphen\textsf{RG}\adap{#1}}
\newcommand{\cRG}[1]{\textsf{RG}\adap{#1}}

\makeatletter
\newcommand\subsubsubsection{\@startsection{paragraph}{4}{\z@}
    {-3.25ex\@plus -1ex \@minus -.2ex}
    {1.5ex \@plus .2ex}
    {\normalfont\normalsize\bfseries}}
\makeatother

\newtheoremstyle{propositionstyle}
  {\topsep}
  {\topsep}
  {\normalfont}
  {}
  {\bfseries}
  {.}
  {\newline}
  {\thmname{#1}\thmnumber{ #2}\thmnote{ (#3)}}

\tikzset{
math to/.tip={Glyph[glyph math command=rightarrow]},
regepi/.tip={Glyph[glyph math command=rightarrowtriangle]},
monoreg/.tip={Glyph[glyph math command=triangleleft]}
}

\theoremstyle{propositionstyle}
\newtheorem{Proposition}[subsubsection]{Proposition}

\newtheorem{Remark}[subsubsection]{Remark}
\newtheorem{Lemma}[subsubsection]{Lemma}

\newtheorem{Example}[subsubsection]{Example}
\newtheorem{Theorem}[subsubsection]{Theorem}
\newtheorem{Definition}[subsubsection]{Definition}
\newtheorem*{Proof}{Proof}

\theoremstyle{thmstyleone}

\theoremstyle{thmstyletwo}

\theoremstyle{thmstylethree}

\begin{document}

\title[Article Title]{Internal structures in the category of right-preordered groups}
\author*[1,2]{\fnm{Ony} \sur{Aubril}}\email{ony.aubril@uclouvain.be}
\affil{\orgdiv{Institut de Recherche en Mathématiques et Physique}, \orgname{Université Catholique de Louvain}, \orgaddress{\city{Louvain-la-Neuve}, \postcode{1348}, \country{Belgium}}}
\abstract{We show that the category of right-preordered groups is a quasivariety of universal algebras by giving explicit axioms, and give a similar description for the category of preordered groups. We then look at lattices of effective equivalence relations, which turn out to be similar to the lattices of equivalence relations in the category of groups. We study internal structures in the category of right-preordered groups, and we especially consider the class $\textsf{S}$ of Schreier split epimorphisms. The category of right-preordered groups turns out to be action representable when we restrict our attention to split epimorphisms in $\textsf{S}$. Relatively to this class of split epimorphisms, we define the notion of $\textsf{S}\mhyphen$precrossed modules, and then of $\textsf{S}\mhyphen$crossed modules that correspond respectively to internal reflexive graphs and internal categories whose underlying split epimorphism is a Schreier split epimorphism. Lastly, we characterize groupoids among those internal categories and give some examples to illustrate the diversity of internal categories in the category of right-preordered groups.}




\maketitle

\section*{Introduction}

A right-preordered is a group endowed with a preorder (i.e. a transitive and reflexive relation) over its underlying set such that the group operation is monotone on the right, i.e. a group $G$ together with a preorder $\leq $ such that for all $g,h,k$ in $G$, $g\leq h\ri g+k\leq h+k$.  Right-preordered groups appear in the literature as early as 1942 in \cite{Orderedgroups}, where F.W. Levi investigated quotients of totally preordered groups. Later, in the context of Fourier analysis,  W. Rudin introduced totally ordered abelian groups in \cite{FourierAnalysisOnGroups}, and proved, most notably, that any archimedean, i.e. such that for all $x\geq y\geq 0$, exists at least one natural integer $n$ such that $ny\geq x$, totally ordered abelian group is isomorphic (via order-preserving isomorphism) to a subgroup of the field of real numbers. Later, the notion of left-ordered groups (which is equivalent to the notion of right-ordered groups), became important since continuous order-preserving actions of a given group $G$ on the real line were able to be associated to left-preorders over $G$ (see \cite{OrderedGroupsAndTopology} for instance). Right-preordered groups also appeared in \cite{PartiallyOrderedGroups}, where some preorders are described as compatible only on one side for some groups.  

More recently, in \cite{RPOgrpCatPersp}, M.M. Clementino and A. Montoli investigated the category $\crpogrp$ of right-preordered groups and monotone group morphisms. This work follows a similar article they wrote about the category $\cpogrp$ of preordered groups, which are groups endowed with a preorder compatible on both sides with the group law (\cite{OCatBehaPOGrp}). They proved that the category of right-preordered groups (similarly to the category of preordered groups), does not share most of the \textquotedblleft good\textquotedblright  algebraic properties with the category $\cgrp$ of groups. For instance, the category is neither Bourn-protomodular \cite{Normalizationequivalencekernelequivalenceandaffinecategories} nor Barr-exact \cite{ExactCategories}. Consequently, they looked into understanding split extensions (for some fixed kernel and codomain), and found both examples of some pairs of right-preordered groups $(K,C)$ that could not induce a split extension with kernel $K$ and codomain $C$, and pairs such that there were an uncountable infinity of such split extensions (example 4.6 in \cite{RPOgrpCatPersp}). For the category $\cpogrp$, M.M. Clementino and C. Ruivo thoroughly explored split extensions in \cite{Onsplitextensionsofpreorderedgroups}, showing that the set of compatible preorders on a semi-direct product of groups is a complete lattice, with as top element the lexicographic preorder and as bottom element the product preorder. Moreover, even though the forgetful functor from $\crpogrp$ to the category $\cgrp$ of groups is topological, the algebraic behaviour of $\crpogrp$ is closer to the behaviour of the category $\cmon$ of monoids that it is to the category of groups. In view of this observation, as well as the characterization of central extensions of preordered groups made by M. Gran and A. Michel in \cite{CentrExtPOGrp}, it becomes natural to turn our attention to a special case of split epimorphisms. Indeed, in both \cite{CentrExtPOGrp} and \cite{Patchkoria}, a similar class of split epimorphisms appear. In \cite{Patchkoria}, A. Patchkoria considers the class of so-called Schreier split epimorphisms in the category $\cmon$, which corresponds to the classical notion of internal monoid action via endomorphisms.

In \cite{RPOgrpCatPersp}, M.M. Clementino and A. Montoli proved that both categories $\crpogrp$ and $\cpogrp$, the full subcategory of preordered groups, are quasivarieties by exhibiting a regular projective, regular generator object in both categories (Proposition A.1. in \cite{RPOgrpCatPersp}), and then with abstract arguments they showed that both quasivarieties are finitary quasivarieties (proposition A.2 in \cite{RPOgrpCatPersp}). However, it seems that no axiomatization has been given for any of the categories $\crpogrp$ or $\cpogrp$. The first section of this article is dedicated to the description of both categories $\crpogrp$ and $\cpogrp$ as quasivarieties of universal algebras. In particular we find an axiomatization $(\Sigma, \Hh)$ (see section \ref{Quasivarieties}) that makes explicit the fact that the category of monoid morphisms such that the codomain is a group is itself a variety of universal algebras, and injectivity of said monoid morphism is the only quasi-identity in $\Hh$. Considering the well-known description of $\crpogrp$ as the category of groups together with a fixed submonoid (see \cite{RPOgrpCatPersp}), the theory $(\Sigma, \Hh)$ induces an equivalence of categories between the category of $(\Sigma, \Hh)\mhyphen$algebras and $\crpogrp$. We finish the section by showing that $\cpogrp$ is not a subquasivariety of $\crpogrp$, but can still be recovered as category of algebras for a theory $(\Sigma',\Hh')$, with $\Sigma\subseteq \Sigma'$ and $\Hh\subseteq \Hh'$. Better still, $\Sigma$ and $\Sigma'$ differ by only one element, and $\Hh$ differs from $\Hh'$ by only two identities. It results that axiom \textbf{(Inj)} in section \ref{Quasivarieties}, reflecting the injectivity of previously mentioned morphism, is the only proper quasi-identity in both $\Hh$ and $\Hh'$,

With this result established, the framework for the study of right-preordered groups becomes clearer for the second section. In particular, the description of kernel pairs becomes more intuitive (see \ref{effequivcomefromgrp}), and we are then able to investigate the links between lattices of normal subobjects in the category $\cgrp$ and normal subobjects in the category $\crpogrp$. For a given right-preordered group $(G, P_G)$, its lattice of normal subobjects in $\crpogrp$ turns out to be isomorphic to the lattice of normal subobjects of $G$ in $\cgrp$. This result is crucial. Indeed, the lattice of normal subobjects in $\cgrp$ is a modular lattice, thus the lattice of normal subobjects of $(G, P_G)$ is also a modular lattice, and because $\crpogrp$ is a normal category in the sense of Z. Janelidze \cite{Thepointedsubobjectfunctor} (as shown in \cite{RPOgrpCatPersp}), the lattice of effective equivalence relations inherits the aforementioned property as well. Considering the work of M. Gran and D. Bourn in \cite{CatAspOfModula}, this result ensures the uniqueness of the composition map over any reflexive graph (note that it is however necessary to be careful and revisit the proof given in \cite{CatAspOfModula}, because we have modularity of the lattice of effective equivalence relations over any right-preordered group, but it is not the case that the lattice of equivalence relations over an arbitrary right-preordered group is a modular lattice). This consideration also gives a natural example of proper relatively modular quasivariety in the sense of \cite{CommTheForRelModQuasiV}, fitting for some further investigations about commutator theory. 

In the fourth section, we shortly give general statements about internal categories and internal groupoids in $\crpogrp$, giving a broad vision of what we could expect to find in terms of internal structures and leading the way for the fifth section, which holds the more original results.

In the fifth section, we will present some results about internal structures  in the category $\crpogrp$. We start by proving the protomodularity of $\crpogrp$ relatively to some class $\textsf{S}$ of split epimorphisms, closely related to Schreier split epimorphisms in $\cmon$. In doing so, we will actually be proving a stronger result, and this is one of the main result of this article : relatively to $\textsf{S}$, the category $\crpogrp$ is action representable (theorem \ref{SActRepr}). We describe the split extension classifier (still relatively to $\textsf{S}$), and thus give a clear definition of what an internal action of right-preordered groups is. This result is to be compared to some other recent developments that have been made in the domain, namely the split extensions classifier that was discovered by M.M. Clementino and A. Montoli in \cite{AlgebraicexponentiationandactionrepresentabilityforVgroups} for all $V\mhyphen$groups with $V$ a cartesian quantale (in particular, for preordered groups), relatively to some class of split epimorphisms. The category $\crpogrp$ becomes another distinct example for the increasing list of categories that are protomodular and action representable relatively to a specific class of split extensions. 

Once we have further investigated internal actions, we take inspiration from the work of A. Patchkoria in \cite{Patchkoria} and we work towards a definition of crossed module of right-preordered groups, still relative to the class $\textsf{S}$. In order to do so, we first define $\textsf{S}\mhyphen$precrossed modules of right-preordered groups, which correspond to internal reflexive graphs in $\crpogrp$ whose \textquotedblleft domain" morphism is a Schreier split epimorphism. We then define $\textsf{S}\mhyphen$crossed modules, which correspond to internal categories whose underlying reflexive graph has as \textquotedblleft domain" morphism a Schreier split epimorphism. Lastly, we characterize groupoids among these internal categories. The proof heavily relies on the usual construction of (pre)crossed modules of groups. As such, it will be clear that any of the above mentioned correspondence is an equivalence of categories. This will be the object of section \ref{RG}.

We then recall the relative Smith-is-Huq condition, as defined in \cite{OntheSmithisHuqconditioninSprotomodularcategories}. We prove that even though the Smith-is-Huq condition does not hold for $\crpogrp$ (example \ref{NotSmithisHuq}), the version relative to the class $\textsf{S}$ does hold. 

We finish the article with some examples to show the diversity of internal categories in $\crpogrp$ when we do not restrict ourselves to split epimorphisms in $\textsf{S}$. With increasing complexity, we are able to provide examples of groupoids that are \textsf{S}-internal categories, but also internal categories that are not groupoids, groupoids that are not \textsf{S}-internal categories, and lastly, internal categories that are neither a groupoid nor an \textsf{S}-internal category. 

\section*{Preliminaries}
\subsection{The category of right-preordered groups}

	A \textit{right-preordered group} is a (not necessarily abelian) group $(G,0,+)$, together with a preorder $\leq$ over the set $\vert G\vert $ such that for any $g,h,k$ in $G$, $g\leq h$ implies $g+k\leq h+k$. For example, any preordered group, as studied in \cite{OCatBehaPOGrp} by M.M. Clementino and A. Montoli, is in particular a right-preordered group. As observed in \cite{RPOgrpCatPersp}, this last category is isomorphic to the full subcategory of monomorphisms $P_G\rightarrowtail G$, in the category $\cmon$ of monoids and monoid morphisms, such that $G$ is a group. Under this equivalence of categories, the category of preordered groups becomes isomorphic to the category of groups $G$, together with a submonoid $P_G\subseteq G$ (equivalently a monomorphism in $\cmon$) that is stable under conjugation by any element of $G$.  In either one of these categories, we shall write objects as pairs $(G,P_G)$, and view $P_G$ as a submonoid of the group $G$. We usually call $P_G$ the \textit{positive cone} of $G$. A morphism in any of these two categories will be written $(f,\overline{f})$, where $f$ is a group morphism and $\overline{f}$ is the restriction of $f$ to the positive cones. 

The category $\crpogrp$ is a quasivariety of universal algebras (Proposition A.1. in \cite{RPOgrpCatPersp}), which makes it a regular complete and cocomplete category. Moreover, the forgetful functor from $\crpogrp$ to $\cgrp$ sending a right-preordered group $(G, P_G)$ to $G$ is topological (\cite{RPOgrpCatPersp}, proposition 2.5.), thus has both a left and a right adjoint. The other forgetful functor, from $\crpogrp$ to the category of preordered sets, turns out to be monadic (\cite{RPOgrpCatPersp}, proposition 2.5.). Considering both these facts, we have a good descriptions of limits, the group structure of a limit is computed as in the category $\cgrp$ of groups and group morphisms, and the preorder is computed as in the category $\cpreord$ of preordered sets and monotone maps. We also recall that epimorphisms in $\crpogrp$ are morphisms $(f,\overline{f})$ such that $f$ is surjective, and regular epimorphisms are exactly epimorphisms as previously such that $\overline{f}$ is also surjective over the positive cone. The category $\crpogrp$ also admits a zero object, which is the trivial group endowed with the trivial preorder. However, $\crpogrp$ fails to be a Barr-exact category (remark 2.6. in \cite{RPOgrpCatPersp}), $\crpogrp$ is therefore a proper quasivariety and not a variety of universal algebras. Still, it was proven that $\crpogrp$ is a normal category in the sense of \cite{Thepointedsubobjectfunctor} (\cite{RPOgrpCatPersp}, proposition 2.12.). 

\subsection{Precrossed and Crossed modules of groups}\label{PXXModGrp}

The lack of protomodularity in $\crpogrp$ is crucial, but following the works of A. Patchkoria in \cite{Patchkoria} for monoids, we will investigate a class of split epimorphisms $\textsf{S}$, which gives us some weaker notion of protomodularity, relatively to the class $\textsf{S}$ (\ref{SemiDirect}). To do this, we will rely on results in the category $\cgrp$, for which we briefly recall the main notions. The category $\cgrp$ is one of the main example of semi-abelian category (see \cite{SemiAbelianCategories}). As such, it is protomodular, and is suitable for a notion of internal crossed module. Crossed modules of groups, first introduced by J.C. Whitehead in \cite{CombinatorialhomotopyII}, have been defined in many other algebraic contexts (such as Lie algebras, commutative algebras, cocommutative Hopf algebras). In \cite{InternalCrossedModule}, G. Janelidze introduced a categorical notion of crossed module in any semi-abelian category.

To fully understand crossed modules of groups, we first recall the proof for the protomodularity of $\cgrp$.

Given in the category $\cgrp$ a split epimorphism with codomain $G$ and given kernel $K$, depicted as 

\begin{center}
\begin{tikzcd}
K\ar[>->, "i",r]&X\ar[r, -regepi, "p", shift left=4pt]&G\ar[l, >->, "s"],
\end{tikzcd}
\end{center}

we get an isomorphism of groups $\phi: X\rightarrow K\rtimes_{\conj{s}} G$ given by $\phi(x)=(x-sp(x), p(x))$ with inverse $(k,g)\mapsto k+s(g)$, where $i$ has been treated as an inclusion, and $\conj{s}$ denotes the automorphism $k\in K\mapsto s+k-s\in K$.  It is well known that this group isomorphism induces an isomorphism of split extensions in $\cgrp$ :

\begin{center}
\begin{tikzcd}
K\ar[>->, "i",r]\ar[d, equal]&X\ar[d, "\phi"]\ar[r, -regepi, "p", shift left=4pt]&G\ar[l, >->, "s"]\ar[d, equal]\\
K\ar[>->, "i_0",r]&K\rtimes_{\conj{s}}G \ar[r, -regepi, "p_1", shift left=4pt]&G\ar[l, >->, "i_1"]
\end{tikzcd}
\end{center}

The validity of the Split Short Five Lemma, which is equivalent to the property of protomodularity in the pointed context, holds in the category $\cgrp$ of groups and it is a consequence of the isomorphism above. If we write $\cPt{K,X}$ for the set of isomorphism classes of split extensions with kernel $K$ and codomain $X$, we actually get a bijection $\cPt{K,X}\cong \cgrp(K,\Aut{X})$. This bijection extends to an isomorphism of functors $\cPt{\_,X}\cong \cgrp(\_,\Aut{X})$, which is a property known as $\textit{action representability}$ (see \cite{Ontherepresentabilityofactionsinasemiabeliancategory} for instance). Aside from the category $\cgrp$, the other classical example of action representable category is the category $\cLie{k}$ of Lie algebras over any field $k$ and Lie algebras morphisms, for which the isomorphism $\cPt{\_,X}\cong \cLie{k}(\_,\Der{X})$ holds, where $\Der{X}$ is the Lie algebra of derivations of $X$.  

Adding some structures to split epimorphisms, we will also look at internal reflexive graphs. A reflexive graph in a category $\CC$ is a diagram 

\begin{center}
\begin{tikzcd}
X\ar[r, -regepi, "d", shift left=4pt]\ar[r, -regepi, "c"', shift right=4pt]&G\ar[l, >->, "e" description],
\end{tikzcd}
\end{center}

in $\CC$ where the morphism $e$ splits both regular epimorphisms $c$ and $d$. This notion of reflexive graph extends to a category $\cRG{\CC}$ by considering morphisms in $\cRG{\CC}$ as pairs of morphisms $(f_0,f_1)$ in $\CC$, making obvious squares commutative. It is well known that the category $\cRG{\cgrp}$ is equivalent to the category $\cpxmod{\cgrp}$ of precrossed modules of groups. In this last category, objects are quadruplets $(G,H,\mu: G\rightarrow \Aut{H}, \partial: H\rightarrow G)$, where $G,H$ are groups and $\mu, \partial$ are group morphisms such that for all  $g\in G$, $h\in H$, $\partial(\mu(g)(h))=g+h-g$. A morphism $f:(G,K,\mu, \partial)\rightarrow (G',K',\mu', \partial')$ is a pair of group morphisms $(f_0:G\rightarrow G', f_1:K\rightarrow K')$ such that $f_0\partial=\partial' f_0$ and for all $g\in G, h\in K, f_0(\mu(h)(g))=\mu'(f_1(h))(f_0(g))$.  The equivalence between $\cRG{\cgrp}$ and $\cpxmod{\cgrp}$ is as follows. With a precrossed module of groups $(G,K, \mu, \partial)$, we associate the reflexive graph 

\begin{center}
        \begin{tikzcd}
|[alias=HM]|K\rtimes_\mu G &|[alias=HR]|G.\\    
\arrow[-regepi, from=HM, to=HR, shift left=4, "p_1"]
    \arrow[-regepi, from=HM, to=HR, shift right=4, "\partial p_1+p_0"']
    \arrow[>->, from=HR, to=HM, shift left=0, "i_1" description]
        \end{tikzcd}
\end{center}

With a morphism of precrossed module of groups $(f_0:G\rightarrow G', f_1:K\rightarrow K')$, we associate 

\begin{equation}
        \begin{tikzcd}
|[alias=HM]|K\rtimes_\mu G \ar[d, "f_1\times f_0"']&|[alias=HR]|G\ar[d, "f_0"]\\    
|[alias=HMM]|K'\rtimes_{\mu'}G' &|[alias=HRR]|G'\\
\arrow[-regepi, from=HM, to=HR, shift right=4, "p_1" description]
    \arrow[-regepi, from=HM, to=HR, shift left=4, "\partial p_0+p_1"]
    \arrow[>->, from=HR, to=HM, shift left=0, "i_1" description]
    \arrow[-regepi, from=HMM, to=HRR, shift left=4, "p_1'" description]
    \arrow[-regepi, from=HMM, to=HRR, shift right=4, "\partial'p_0+p_1"']
    \arrow[>->, from=HRR, to=HMM, shift left=0, "i_1" description]
        \end{tikzcd}\label{equivalencemorphismsRG}
   \end{equation}

Conversely, with any reflexive graph 
\begin{center}
        \begin{tikzcd}
|[alias=HM]|X&|[alias=HR]|G\\ 
\arrow[-regepi, from=HM, to=HR, shift left=4, "d"]
    \arrow[-regepi, from=HM, to=HR, shift right=4, "c"']
    \arrow[>->, from=HR, to=HM, shift left=0, "e" description]
        \end{tikzcd}
   \end{center}

we associate the precrossed module of groups $(G,\ker{(d)}, \conj{e}, c_{\vert \ker{(d)}})$ and a morphism $(f_0,f_1)$ is associated to $(f_0, {f_1}_{\vert\ker{(d)}})$.

This equivalence restricts well to crossed modules of groups and internal categories. Since the category $\cgrp$ is a Mal'tsev category (example 2.2.4. in \cite{Malcevprotomodularhomologicalandsemiabeliancategories}), being an internal category for an internal reflexive graph is a property, not a structure. Crossed modules of groups are exactly precrossed modules $(G,K, \mu, \partial)$ of groups such that the so-called Peiffer identity holds : 

\begin{align*}
\textbf{(P)}\quad \forall k,l\in K, \mu(\partial(k))(l)=k+l-k
\end{align*}

Most importantly the corresponding reflexive graph is an internal category exactly when the corresponding precrossed module of groups is a crossed module of groups. One can find a detailed version of the equivalence of categories between the category $\cCat(\cgrp)$ of internal categories and internal functors in $\cgrp$ and the category of crossed modules of groups in\cite{Categoriesfortheworkingmathematician}, chapter XII., section 8.

\subsection{The Smith-is-Huq property}

 Another related property of the category $\cgrp$ that we will also investigate in the category $\crpogrp$ is the so-called \textquotedblleft Smith-is-Huq" property. Existence of a composition morphism for an internal reflexive graph is linked to some notion of centrality, namely the Smith-centrality of equivalence relations, or, equivalently in $\cgrp$, the Huq-commutativity of corresponding normal monomorphisms. We recall both definitions of Smith-centrality and Huq-commutativity. 

The notion of Huq-commutativity makes sense in the unital context. We recall the definition. 
\begin{Definition}[Unital category, \cite{Malcevcategoriesandfibrationofpointedobjects}]
    A category $\CC$ is said to be \textit{unital} if $\CC$ is pointed, finitely complete, and for all objects $X,Y$ in $\CC$, the pair of morphisms $(\id{X}\times 0:X\to X\times Y, \;0\times \id{Y}:Y\to X\times Y)$ is jointly strongly epimorphic.
\end{Definition}

\begin{Definition}[Huq-commutativity, \cite{Commutatornilpotencyandsolvabilityincategories}]
    We say that two subobjects $X,Y$ of an object $Z$ \textit{commute in the sense of Huq} in a unital category $\CC$ if there exists a dashed arrow as below such that both triangles commute. 
    
    \begin{center}
        \begin{tikzcd}
            X\ar[>->, "\id{X}\times 0", r]\ar[>->, dr]&X\times Y\ar[d, dashed ]&Y\ar[>->, "0\times \id{Y}"', l]\ar[>->, dl]\\
            &Z&
        \end{tikzcd}
    \end{center}
\end{Definition}

In the category $\cgrp$ it is well known that such a morphism exists if and only if any element of $X$ commute with any element of $Y : xy=yx$ for any $x\in X,y\in Y$. The dashed arrow is then the group law restricted to subgroups $X$ on the left and $Y$ on the right. 

Since in the definition $\CC$ is unital, the dashed morphism is unique if it exists, therefore, for $X,Y$ subobjects of $Z$ in $\CC$, commuting in the sense of Huq is a property.

\begin{Definition}[Smith-centrality, \cite{Malcevvarieties}, \cite{Acategoricalapproachtocommutatortheory}]
    In a finitely complete category $\CC$, we say that two reflexive relations $R,S$ over the same object $X$ \textit{centralize each other in the sense of Smith} if in the pullback below there exists an arrow $p$ such that $p\Delta_1=p_1^S$ and $p\Delta_0=p_0^R$.

    \begin{center}
        \begin{tikzcd}
            R\times_X S \ar[shift left=2, "p_1", r, -regepi]\ar[shift left=2, "p_0", d, -regepi]&S\ar[shift left=2, "\Delta_1", l, >->]\ar[d, -regepi, "p_0^S", shift left=2]\\
            R\ar[shift left=2, "p_1^R", r, -regepi]\ar[u, >->, "\Delta_0", shift left=2]&X\ar[l, >->, "\Delta_R", shift left=2]\ar[u, >->, "\Delta_S", shift left=2]
        \end{tikzcd}
    \end{center}
\end{Definition}

\begin{Definition}[Normalization of an equivalence relation, \cite{Normalizationequivalencekernelequivalenceandaffinecategories}]
In a finitely complete pointed category $\CC$, the \textit{normalization} of an equivalence relation $R$ over an object $X$ is the subobject $N$ of $X$, such that the following diagram is a pullback.

\begin{center}
        \begin{tikzcd}
           N\ar[d, >->]\ar[r,->]& R\ar[d, >->]\\
	X\ar[r, "(0\mathpunct{{,}}\id{X})"']&X\times X
        \end{tikzcd}
\end{center}
\end{Definition}

It is always true that if two equivalence relations $R,S$ centralize each other in the sense of Smith, then their normalization commute in the sense of Huq (see \cite{CentralityandNormalityinprotomodularCategories}). However the converse is not true in general. When it is true in a category $\CC$, we say that $\CC$ satisfies the \textit{Smith-is-Huq} property, it is the case of the category $\cgrp$, but fails to be the case for $\crpogrp$ (see example \ref{NotSmithisHuq}).

\section{The quasivariety of (right-)preordered groups}\label{Quasivarieties}
\subsection{Algebraic theory of right-preordered groups}

Without giving an explicit axiomatization, M.M. Clementino and A. Montoli proved that both categories $\cpogrp$ and $\crpogrp$ are finitary quasivarieties of universal algebra (\cite{RPOgrpCatPersp}, theorem A.2 and remark A.3.). We give an example of explicit axiomatization. 

First, we focus on right-preordered groups. To axiomatize this quasivariety, we consider a right-preordered group as a set $G\times P_G$, together with a monoid structure $((0_G, 0_{P_G}),+)$. To reflect the cartesian product of the set, we ask the set of operations to contain two projections $\pym_0$ and $\pym_1$ to respectively $G\times \{0_{P_G}\}$ and $\{0_G\}\times P_G$. Both $\pym_0$ and $\pym_1$ are monoid morphisms, so we recover a monoid structure on both $\pym_0(G\times P_G)$ and $\pym_1(G\times P_G)$. To get a group structure over $\pym_0(G\times P_G)$ we add an operation $-$ that acts as the inverse operation over $\pym_0(G\times P_G)$ and is \textquotedblleft somehow orthogonal" to $\pym_1$. Lastly, we ask for an injection $\iym$ from $\pym_1(G\times P_G)$ to $\pym_0(G\times P_G)$ that is also a monoid morphism. Asking $\iym$ to be injective is actually the only quasi-identity that we get in this axiomatization.  

More formally, we give the following generalized equational Horn theory. The signature $\Sigma$ of the theory of this quasivariety consists of one function of arity zero that we will write $0$, four unary operations $\iym, \pym_0, \pym_1, -$ and one binary operation $+$. We will write $a+b$ for $+(a,b)$, $a-b$ for $a+-b$, $a_i$ for $\pym_i(a)$ $(i=0,1)$. We have the following axioms :

\begin{itemize}
    \item[\textbf{(M)}] $(0,+)$ is the usual theory of monoids. $\pym_0,\pym_1,\iym$ are monoid morphisms ($f(a)+f(b)=f(a+b)$, $f(0)=0$ for $f\in \{\pym_0,\pym_1, \iym$\}).
    \item[\textbf{(P1)}] $a=a_0+a_1=a_1+a_0$.
    \item[\textbf{(P2)}] $-(a_1)=(\iym(a))_1=(a_1)_0=(a_0)_1=0$.
    \item[\textbf{(P3)}] $\iym(a_1)=(\iym(a))_0$.
    \item[\textbf{(G)}] $a-a=-a+a=a_1$.
    \item[\textbf{(Inj)}] $\quad \iym(a)=\iym(b)\ri a_1=b_1$.    
\end{itemize}

This set of quasi-equations will be denoted by $\Hh$. 

The category of models of this theory in the category $\cset$ will be written $\cAlg{\Sigma,\Hh}$. A model $X$ in $\cAlg{\Sigma, \Hh}$ will be written as $(X,0^X, \pym_0^X, \pym_1^X, -^X, \iym^X, +^X)$.

\begin{Theorem}
    We have the equivalence of categories  $$\cAlg{\Sigma,\Hh}\cong \crpogrp.$$
\end{Theorem}

\begin{proof}
We consider a functor $\Fym: \crpogrp\longrightarrow \cAlg{\Sigma,\Hh}$ as follows. 

With any right-preordered group $\underline{G}=((G,e,s,\star),(P_G,1, *),i)$ (i.e. $i$ a monomorphism in $\cmon$ with domain the monoid $(P_G,1,*)$ and codomain the group $(G,e,s,\star)$, with neutral $e$ and inverse map $s$), we associate the model of $\cAlg{\Sigma,\Hh}$    $\Fym(\underline{G})=(G\times P_G,0^{\underline{G}}=(e,1),\pym_0^{\underline{G}}:(a,b)\mapsto(a,1), \pym_1^{\underline{G}}:(a,b)\mapsto(e,b), -^{\underline{G}}:(a,b)\mapsto (s(a),1),$$ \iym^{\underline{G}}:(a,b)\mapsto (i(b),1), +^{\underline{G}} :((a,b),(c,d))\mapsto(a\star c,b*d))$ in $\cset$. With any morphism $(f,\overline{f})$ we simply associate $f\times \overline{f}$.

If $(f,\overline{f})$ is a morphism in $\crpogrp$ that is left composable with $(g,\overline{g})$, then $\Fym((f,\overline{f})\circ (g, \overline{g}))=\Fym(f\circ g, \overline{f}\circ \overline{g})=(f\circ g)\times (\overline{f}\circ \overline{g})=(f\times \overline{f})\circ (g\times \overline{g})=\Fym(f,\overline{f})\circ \Fym(g,\overline{g})$ so $\Fym$ is a functor. 
We check $\Fym(G)$ is an object in $\cAlg{\Sigma,\Hh}$. We chose a different symbol for all operations so we shall lose the exponent $\underline{G}$ for clarity. We write $(a,x), (b,y)$ for two arbitrary elements of $G\times P_G$

\begin{itemize} 
    \item[\textbf{(M)}]  
    \begin{itemize}
        \item $(G\times P_G, (0,1), +)$ is a monoid, because it is the product  in $\cmon$ of two monoids,
        \item $\pym_0, \pym_1$ are monoid morphisms, being the product of canonical projections with the trivial morphism in  $\cmon$,
        \item $\iym$ is a monoid morphism because it respects the neutral and  \begin{align*}\iym((a,x)+(b,y))=&\iym(a\star b,x*y)\\&=(i(x*y),1)\\&=(i(x)\star i(y),1)\\&=(i(x),1)+ (i(y),1)\\&=\iym(a,x)+ \iym(b,y).\end{align*}
    \end{itemize}
    \item[\textbf{(P1)}] This is immediate because the operation $+$ is componentwise and $1$ commutes with every element of $P_G$, $e$ commutes with every element of $G$.     
    \item[\textbf{(P2)}] 
    \begin{itemize}
        \item $-\pym_1(a,x)=-(e,x)=(s(e),1)=(e,1)$,
        \item $\pym_1\iym(a,x)=\pym_1(i(x),1)=(e,1)$,
        \item $\pym_0\pym_1(a,x)=\pym_0(e,x)=(e,1)$,
        \item $\pym_1\pym_0(a,x)=\pym_1(a,1)=(e,1)$.
    \end{itemize}
    \item[\textbf{(P3)}] $ \iym \pym_1(a,x)=\iym(e,x)=(i(x),1)=\pym_0(i(x),1)= \pym_0\iym(a,x)$.
    \item[\textbf{(G)}] $(a,x) + (-(a,x))=(a,x)+ (s(a),1)=(e,x*1)\pym_1(a,x)$, and the other equality is analogous. 
    \item[\textbf{(Inj)}] $\iym(a,x)=\iym(b,y)\ri (i(x),1)=(i(y), 1)\ri x=y\ri \pym_1(a,x)=(e,x)=(e,y)=\pym_1(b,y)$ because $i$ is injective.    
\end{itemize}

Now let $(f,\overline{f}):((G,e_G, s_G, \star_G),(P_G, 1_G,*_G),i_G)\longrightarrow ((H,e_H,s_H, \star_H),(P_H,1_H,*_H),i_H)$ be a morphism in $\crpogrp$. We check $f\times \overline{f}$ is a morphism in $\cAlg{\Sigma, \Hh}$. 
\begin{itemize}
    \item The domain of $f\times\overline{f}$ is $G\times P_G$ and the codomain is $H\times P_H$ because $\overline{f}$ is the restriction of $f$ to positive cones,
    \item by the property of the product in the category of monoids, $f\times \overline{f}$ is a morphism of monoids,
    \item $(f\times \overline{f})\pym_0=\pym_0(f\times \overline{f})$, because in both case we get $f$ on the first component and the terminal map of $\cmon$ on the second one. The same holds for $(f\times \overline{f})\pym_1=\pym_1(f\times \overline{f})$,
    \item $(f\times \overline{f})(-)=(fs_G)\times 0=(s_H f)\times 0=-(f\times \overline{f})$ because $f$ is a group morphism, 
    \item $\forall (a,x)\in G\times P_G$, $(f\times \overline{f})\iym (a,x)=(fi_G(x), \overline{f}(1_G))=(i_H\overline{f}(x), 1_H)=\iym(f(a),\overline{f}(x))=\iym(f\times \overline{f})(a,x)$. 
\end{itemize}

Conversely we define the functor $\Gym : \cAlg{\Sigma, \Hh} \longrightarrow \crpogrp$ as follows :
with any model $X$, we associate the group $(\pym_0^X(X),0^X,-^X,+^X)$ together with a submonoid $(\pym_1^X(X),0^X,+^X)$, that is included in $\pym_0^X(X)$ via the operation $\iym^X_{\vert \pym_1^X(X)}:\pym_1^X(X)\rightarrow \pym_0^X(X)$. By $+^X$ we mean the obvious restriction of the operation $+^X$ of $X$ to $\pym_0^X(X)$ and $\pym_1^X(X)$ respectively. A morphism $f$ between two models $X$ and $Y$ is associated to $(f_{\vert \pym_0^X(X)},f_{\vert \pym_1^X(X)})$. We write $X_0:=\pym_0^X(X)$, $X_1:=\pym_1^X(X)$, $i^X:=\iym^X_{\vert \pym_1^X(X)}$, and lose exponents $X$ for readability.  
We prove that $\Gym(X)$ is a right-preordered group. 

First, $(X_0, 0, -, +)$,  $(X_1, 0, +)$ are monoids, being direct images of the monoid $(X,0, +)$ via monoid morphisms $\pym_0^X$ and $\pym_1^X$ respectively. Moreover,  for any $x$ in $ X_0, x+(-x)=(-x)+x=x_1=0$, because $x\in X_0$ and by axiom $\textbf{(P2)}$, thus $X_0$ is a group. Lastly,  
$i$ is a monoid morphism from $X_1$ to $X_0$ by definition and is injective ; for any $x,y$ in $X_1, i(x)=i(y)\ri x_1=y_1\ri x=y$ (Axiom $\textbf{(Inj)}$).
   
Now suppose $f$ is a morphism in $\cAlg{\Sigma,\Hh}$ from $X$ to $Y$. We prove $(f_0, f_1):=\Gym(f)$ is a morphism in the category $\crpogrp$. 
\begin{itemize}
    \item $f_0\pym_0^X=\pym_0^Yf_0$ so $f_0$ is a group morphism from $X_0$ to $Y_0$, being the restriction of a monoid morphism to a group with image in $Y_0$, also a group.
    \item $f_1\pym_1^X=\pym_1^Y f_1$ so $f_1$ is a monoid morphism from $X_1$ to $Y_1$. 
    \item $i^Yf_1=\iym^Yf_1=f\iym^X_{\vert \pym_1^X(X)}=fi^X=f(\pym_0^Xi^X+\pym_1^Xi^X)\Delta^X_{\vert\pym_1^X(X)}=f\pym_0^Xi^X=f_0i^X$, where $\Delta^X:X\rightarrow X\times X$ is the diagonal morphism. The fourth equality here comes from axiom $\textbf{(P1)}$.
\end{itemize}

Moreover if $f, g$ are two composable morphisms in $\cAlg{\Sigma,\Hh}$, $\Gym(f\circ g)=((f\circ g)_0, (f\circ g)_1)=(f\circ g_0, f\circ g_1)=(f_0\circ g_0, f_1\circ g_1)=\Gym(f)\circ \Gym(g)$ because for any morphism $h:A\rightarrow B$ in $\cAlg{\Sigma, \Hh}$, the diagram (\ref{picommute}) has to commute for $i\in \{0,1\}$.

\begin{equation}
\begin{tikzcd}
    A\ar[r,"h"]\ar[d,"\pym_i^A"']&B\ar[d,"\pym_i^B"]\\
    A\ar[r,"h"]&B
\end{tikzcd}\label{picommute}
\end{equation}

This proves $\Gym$ is a functor.

We prove that the composites of these two functors $\Fym$ and $\Gym$ are isomorphic to the identity functors.

$$
\func{\Fym\circ \Gym}{\cAlg{\Sigma,\Hh}}{\cAlg{\Sigma,\Hh}}{X}{X_0\times X_1}{f}{f_0\times f_1}
$$

The functor $\Fym\circ \Gym$ is isomorphic to the identity over $\cAlg{\Sigma,\Hh}$ via the natural transformation given over any object $X_0\times X_1$ by the morphism 

\begin{center}
$\mapp{m_X}{X_0\times X_1}{X}{(a,b)}{a+b,}$
\end{center}

 which has an inverse 

\begin{center}
$\mapp{m_X\mm}{X}{X_0\times X_1}{x}{(x_0,x_1)}$
\end{center}

exactly thanks to axiom $\textbf{(P1)}$. We can check that $m$ is indeed a morphism. Consider a $(\Sigma,\Hh)\mhyphen$algebra $X$. For any $(x_0, y_1)), (z_0,w_1)$ in $X_0\times X_1$,
\begin{align*}
m_X((x_0,y_1)+(z_0,w_1))&=m_X((x+z)_0,(y+w)_1)\\
&=(x+z)_0+(y+w)_1\\
&=x_0+z_0+y_1+w_1\\
&=x_0+y_1+z_0+w_1\\
&=m_X(x_0,y_1)+m_X(z_0,w_1)
\end{align*}
The fourth equality comes from the identity $\textbf{(P1)}$ and $\textbf{(P2)}$, since
\begin{align*}
            z_0+y_1&=(z_0+y_1)_1+(z_0+y_1)_0\\
            &=(z_0+y_1)_0+(z_0+y_1)_1\\
            &=y_1+z_0.
\end{align*}

Next we have to consider
$$
\func{\Gym\circ \Fym}{\crpogrp}{\crpogrp}{(G, P_G)}{(\pym_0(G\times P_G),\pym_1(G\times P_G))}{(f,\overline{f})}{((f\times \overline{f})_0, (f\times \overline{f})_1)),}
$$

which is isomorphic to the identity just by identifying $\pym_0(G\times P_G)$ with $G$, as well as $\pym_1(G\times P_G)$ with $P_G$. 
\end{proof}

\subsection{Algebraic theory for preordered groups}
 The category of preordered groups is not a subquasivariety of the quasivariety of right-preordered groups, because the category of preordered groups is not closed under subobjects in the category of right-preordered groups. To get an example, just consider the symmetric group $\Sym_4$ over four elements. It is a preordered group if we consider the positive cone $\Aym_4$, the alternating group over four elements. $\Aym_4$ is a normal subgroup thus a submonoid closed under $\Sym_4$-conjugation. With this group we have a subobject whose group is still $\Sym_4$ but with positive cone $P=\{\Id,(12)(34)\}$. It is a submonoid since it even is a subgroup of order 2 of $\Sym_4$, but it is not closed under conjugation in $\Sym_4$ because $(13)(12)(34)(13)=(32)(14)\not\in P$. 

As a consequence, the axiomatization of preordered groups can not be done by adding identities and quasi-identities to the theory of right-preordered groups. We can still describe its theory by adding a unary operation expressed by $\cym$ that is \textquotedblleft right-orthogonal" to $\pym_0$ and such that $\iym\cym$ is the conjugation of $\iym \pym_1$ by $\pym_0$. Once again formally this is two new axioms :

\begin{itemize}
    \item[\textbf{(C1)}] $\iym\cym(a)=a_0+\iym(a_1)-a_0$.
    \item[\textbf{(C2)}] $(\cym(a))_0=0$.
\end{itemize}

    We write $\Sigma'=\Sigma\cup\{\cym\}$ and $\Hh'=\Hh\cup\{\textbf{(C1)},\textbf{(C2)}\}$ so that the category of models in the category $\cset$ for all of the axioms of $\cAlg{\Sigma, \Hh}$ together with the operation $\cym$ and the axioms $\textbf{(C1)}$ and $\textbf{(C2)}$ will be written $\cAlg{\Sigma', \Hh'}$. If $X$ is a model of $(\Sigma', \Hh')$ in $\cset$, we will write $\cym^X$ to denote this new operation for $X$. 
    
\begin{Theorem}
    We have an equivalence of categories $$\cAlg{\Sigma', \Hh'}\cong \cpogrp.$$ 
\end{Theorem}
\begin{proof}
    We only need to check that $\cAlg{\Sigma', \Hh'}$ is a full subcategory of $\cAlg{\Sigma, \Hh}$ and then that the objects correspond exactly to the preordered groups via the equivalence above. Let $X,Y$ be objects in $\cAlg{\Sigma', \Hh'}$, and $f:X\rightarrow Y$ be an arrow in $\cAlg{\Sigma, \Hh}$, i.e. we do not require $f$ to commute with $\cym$. 

    We have, for any $a\in X$,
    \begin{align*}
        \iym^Y f\cym^X (a)&=f\iym^X \cym^X (a)\\
        &=f(a_0+^X\iym^X(a_1)-^X a_0)\\
        &=f(a_0)+^Y\iym^Y f(a_1)-^Yf(a_0)\\
        &=f(a)_0+^Y\iym^Y f(a)_1-^Yf(a)_0\\
        &=\iym^Y\cym^Y f(a)
    \end{align*}
    and thanks to axiom $\textbf{(Inj)}$, we can deduce that $\cym^Y f=f\cym^X$. This proves the category $\cAlg{\Sigma', \Hh'}$ is a full subcategory of $\cAlg{\Sigma, \Hh}$. 

    Now if we apply the restriction of $\Gym$ to the theory $\cAlg{\Sigma', \Hh'}$, for an object $X$ in $\cAlg{\Sigma', \Hh'}$, we get a map $\Gym(\cym^X)$. With axiom $\textbf{(C1)}$, we get the identity $\iym^X\Gym(\cym^X)(a_0,a_1)=a_0+\iym^X(a_1)-a_0$. In particular, the image of $\iym^X(X_1)$ is stable under conjugation. This is exactly requiring that $X$ is a preordered group. Conversely for any preordered group $G\times P_G$, there is a unique map $\cym^G$ satisfying $\textbf{(C1)}, \textbf{(C2)}$ that can be defined, namely $\cym^G(a_0,a_1)=(0_G,a_0+a_1-a_0)$ (considering $P_G$ as a subset of $G$). Indeed $\textbf{(C2)}$ forces $(\cym^G(a_0,a_1))_0=0_G$ and axiom $\textbf{(C1)}$ uniquely defines $(\cym^G(a_0,a_1))_1=a_0+a_1-a_0$. This concludes the proof that the categories $\cAlg{\Sigma', \Hh'}$ and $\cpogrp$ are equivalent.     
\end{proof}
The operation $\pym_0$ induces the forgetful functor from $\crpogrp$ (or $\cpogrp$) to the category $\cgrp$ of groups and group morphisms under this equivalence. It is therefore topological in both cases (\cite{RPOgrpCatPersp}, proposition 2.5., \cite{OCatBehaPOGrp}, proposition 2.3.), thus faithful and preserves limits and colimits. 

\section{Some categorical properties of the category of right-preordered groups}

\subsection{Normal subobjects and equivalence relations }

The category of right-preordered group is normal in the sense of \cite{Thepointedsubobjectfunctor}, meaning every kernel pair is the smallest effective equivalence relation generated by the zero class in the kernel pair. A consequence of this is that we can restrict ourselves  to study the lattice of kernels in the category $\crpogrp$, instead of the lattice of effective equivalence relations. It was shown in  \cite{RPOgrpCatPersp} (remark 2.6.) for right-preordered groups that kernels are as follows :

\begin{center}
\begin{tikzcd}
    |[alias=HL]|\ker(\overline{f})= \ker(f)\cap P_G&|[alias=HM]|P_G&|[alias=HR]|P_H\\
    |[alias=LL]|\ker(f) &|[alias=LM]|G&|[alias=LR]|H.\\ 
    \arrow[>->, from=HL, to=LL]
    \arrow[>->, from=HM, to=LM, "i"]
    \arrow[>->, from=HR, to=LR]
    \arrow[from=HM, to=HR, "\overline{f}"]
    \arrow[from=LM, to=LR, "f"']
    \arrow[monoreg->, from=HL, to=HM, "\overline{k}"]
    \arrow[monoreg->, from=LL, to=LM, "k"']
\end{tikzcd}
\end{center}

In the left hand square above the positive cone of the kernel is actually the pullback of $k$ along $i$. The kernel $\ker(f)$ is the usual kernel in the category $\cgrp$, and therefore is a normal subgroup of $G$. From the algebraic theory point of view, this description of effective equivalence relations gives a useful tool to compute directly relations. We give it as a lemma. 

\begin{Lemma}\label{effequivcomefromgrp}
   We write $\cEq{f}$ for the kernel pair of a morphism $f$.  Any effective equivalence relation $\cEq{f}=\cEq{f_0\times f_1}$ over $G\times P_G$ in $\crpogrp$ is uniquely determined by $P_G$ and $\cEq{f_0}$ over $G$ in $\cgrp$. Moreover, for all $(a_0,a_1),(b_0,b_1)$ in $G\times P_G$, $(a_0,a_1)\sim_{\cEq{f}}(b_0,b_1)$ if and only if $i(a_1)\sim_{\cEq{f_0}}i(b_1)$ and $a_0\sim_{\cEq{f_0}}b_0$. 
\end{Lemma}
\begin{proof}
    The first assertion directly comes from the fact that $\pym_0$, seen as a forgetful functor, is topological over $\cgrp$ (\cite{RPOgrpCatPersp}, proposition 2.5).
    
    For the second assertion, suppose that for some $(a_0,a_1), (b_0,b_1)$ in $G\times P_G$, $(f_0(a_0),f_1(a_1))=(f_0(b_0),f_1(b_1))$. It is immediate that then $a_0\sim_{\cEq{f_0}}b_0$. Since $f$ is a morphism in $\crpogrp$, $i f_1(a_1)=f_0i(a_1)$ and $if_1(b_1)=f_0i(b_1)$. From these equalities comes $i(a)\sim_{\cEq{f_0}}i(b)$. Conversely, if $i(a)\sim_{\cEq{f_0}}i(b)$, then by definition $f_0i(a_1)=f_0i(b_1)$. Then, since $f$ is a morphism in $\crpogrp$, $if_1(a_1)=if_1(b_1)$. Lastly, because $i$ is injective, $f_1(a_1)=f_1(b_1)$. If we additionally suppose that $a_0\sim_{\cEq{f_0}}b_0$, then we get $(a_0, a_1)\sim_{\cEq{f}}(b_0, b_1)$.
\end{proof}

A full description of the lattice of kernels has already been given in \cite{CentrExtPOGrp}, lemma 2.4., for preordered groups, the exact same description still holds for right-preordered groups. We now recall it for clarity and completeness.

Given a right-preordered group $(G,P_G)$ and two normal subobjects $(N, N\cap P_G)$ and $(M, M\cap P_G)$, their join is given by

\begin{align*}
    (N,N\cap P_G)\lor (M, M\cap P_G)=(N\cdot M, (N\cdot M)\cap P_G)
\end{align*}

where $N\cdot M=\{n+m\vert n\in N, m\in M\}$ is the usual join of normal subgroups in $\cgrp$. The meet is just the intersection 
\begin{align*}
    (N,N\cap P_G)\land (M, M\cap P_G)=(N\cap M, N\cap M\cap P_G). 
\end{align*}

It was shown in \cite{CentrExtPOGrp} (proposition 2.5.) that for any given preordered group $(G, P_G)$, the lattice $\lnorm{G,P_G}$ is actually modular. Here we give another proof, note that this proof still holds in the category $\cpogrp$ and the proof in \cite{CentrExtPOGrp} still holds in the category $\crpogrp$. 

\begin{Lemma}\label{isolattices}
    Given a right-preordered group $(G, P_G)$, its lattice of normal subobjects $\lnorm{G,P_G}$ is isomorphic to the lattice of normal subgroups $\lnorm{G}$ of the group $G$. 
\end{Lemma}
\begin{proof}
      Consider the bijection 
      \begin{align*}
          (N,P_G\cap N)\in \lnorm{G,P_G}\longleftrightarrow N\in \lnorm{G}
      \end{align*}
      One direction is induced by the forgetful functor $\pym_0: \crpogrp\longrightarrow \cgrp$ and we get the inverse by taking the pullback of the inclusion of $P_G$ into $G$ along the inclusion of $N$ into $G$. It is immediate to see that the join is respected in both directions, as is the meet considering $(M\cap N)\cap P_G=(M\cap P_G)\cap(N\cap P_G)$. This bijection is therefore a lattice morphism, hence a lattice isomorphism. 
\end{proof}

This behavior of normal subobjects is to be understood through the following definition.

\begin{Definition}[Relatively modular quasivariety \cite{CommTheForRelModQuasiV}]
    A quasivariety of universal algebra $\KK$ is said to be \textit{relatively modular} if, for any $\KK$-algebra $A$, the lattice of effective equivalence relations over $A$, with the usual preorder, is a modular lattice. By definition, a lattice $L$ is modular if it satisfies the modular identity $(a\land b)\lor(c\land b)=((a\land b)\lor c)\land b$ for all $a,b,c\in L$.
\end{Definition}

\begin{Proposition}\label{modular}
    $\crpogrp$ is a relatively modular quasivariety. 
\end{Proposition}
\begin{proof}
    This is immediate from the previous lemma since the variety of groups is congruence modular, and $\crpogrp$ is normal in the sense of Z. Janelidze (proposition 2.12 in \cite{RPOgrpCatPersp}). 
\end{proof}

\begin{Remark}
    Even though the lattice of normal subobjects in $\crpogrp$ is isomorphic to the lattice of normal subobjects in $\cgrp$, $\crpogrp$ does not inherit the property of being a Mal'tsev category from $\cgrp$ (see theorem 3.5. in \cite{RPOgrpCatPersp}). Notice we do not even have $R\circ S=S\circ R$ for all effective equivalence relations $R,S$ because it has been shown in \cite{SomeRemMalGoursCat}, theorem 5.7 that the permutability of effective equivalence relation is actually equivalent to the permutability of all equivalence relations in a regular category. Because of this, the composition of two effective equivalence relations in $\crpogrp$ can not be given by the composition in $\cgrp$ and the isomorphism above does not preserve composition. 
\end{Remark}
\begin{Remark}
    Notice that because $\crpogrp$ is not a Barr-exact category (remark 2.6 in \cite{RPOgrpCatPersp}), we need to distinguish between two different joins. We could define a join between $\cEq{f}$ and $\cEq{g}$ as the equivalence relation $\cEq{f}\cup \cEq{f}\circ \cEq{g}\cup\cEq{f}\circ \cEq{g}\circ \cEq{f}\cup...$. This relation fails to be effective in general. To get an effective equivalence relation, we will consider the smallest effective equivalence relation that contains both $\cEq{f}$ and $\cEq{g}$. This is exactly the join we consider (\cite{CentrExtPOGrp}, lemma 2.4).  
\end{Remark}

Another property that the category $\crpogrp$ does not have is to be an ideal determined category (in the sense of \cite{IdealDeterminedCategories}), even though $\cgrp$ is. We recall the definition, and then give a counterexample. 

\begin{Definition}[Ideal determined category, \cite{IdealDeterminedCategories}]
    A pointed finitely complete and cocomplete category $\CC$ is said to be ideal determined if 
    
    \begin{itemize}
        \item There is a pullback-stable factorization system (normal epi, mono), where \textquotedblleft normal epi" means \textquotedblleft cokernel of some morphism" ;
        \item for any commutative diagram 
        \begin{center}
            \begin{tikzcd}
                A\ar[r, "q"]\ar[d, "w"'] &B\ar[d, "v"]\\
                C\ar[r, "p"']&D,
            \end{tikzcd}
        \end{center}

        if $p,q$ are normal epimorphisms and $w$ is a normal monomorphism, then $v$ is a normal monomorphism. 
    \end{itemize}
\end{Definition}

The first point is valid, the factorization system (normal epi, mono) is actually exactly the factorization system (regular epi, mono), and the latter is a pullback-stable factorization system in any quasivariety of universal algebras. The second point is not valid, for example, we could consider the right-preordered groups $(\Zz, \Nn)$, $(\Zz\times \Zz, \Delta(\Nn))$, where $\Delta(\Nn)$ is the diagonal of $\Nn\times \Nn$, and $(\Zz, 0)$. In the following diagram

\begin{equation}
\begin{tikzcd}
(\Zz,0)\ar[r, equal]\ar[d, monoreg->, "\ker{p_0}"']&(\Zz,0)\ar[d,>->]\\
(\Zz\times \Zz, \Delta(\Nn))\ar[r, -regepi, "p_1"]&(\Zz, \Nn),
\end{tikzcd}\label{EXG}
\end{equation}

the arrow on the right is the kernel of the first projection, but the second projection $p_1$, which is a normal epimorphisms, is the identity morphism when restricted to $(\Zz,0)$ the kernel of $p_0$. We get a monomorphism $(\Zz,0)\rightarrowtail(\Zz, \Nn)$ which is not a normal monomorphism because $0$ is not $\Zz\cap \Nn$. 

\subsection{Some general remarks on internal structures in $\crpogrp$}

Before diving into classification of some well-behaved internal structures, let us give some general results about internal categories.

\begin{Proposition}\label{uniquenessm}
    Given a reflexive graph 
    
    \begin{center}
    \begin{tikzcd}
       \RR=& (X_1, P_{1}) \arrow[-regepi, r, shift left=1.5ex, "d"]
      \arrow[-regepi, r, shift right=1.5ex, swap, "c"] \arrow[r, <-<, swap, "e" description]& (X_0, P_{0})
   \end{tikzcd} 
   \end{center}
   
   in $\crpogrp$, there exists at most one morphism $m$ in $\crpogrp$ such that 
   
   \begin{center}
   \begin{tikzcd} (X_2,P_2)\arrow[-regepi, r, shift left=1.5ex, "p_0"]
      \arrow[-regepi,r, shift right=1.5ex, swap, "p_1"] \arrow[-regepi, r, swap, "m" description] &(X_1, P_1) \arrow[-regepi, r, shift left=1.5ex, "d"]
      \arrow[-regepi, r, shift right=1.5ex, swap, "c"] \arrow[r, <-<, swap, "e" description]& (X_0, P_0)
   \end{tikzcd}
   \end{center}
   
   is an internal category, where $X_2=\{(a,b)\in X_1\vert d(a)=c(b)\}$ and $P_2=X_2\cap (P_1\times P_1)$, such that $((X_2,P_2), p_0,p_1)$ is the pullback of $d$ along $c$.    
\end{Proposition}
\begin{proof}
    Let $(a,b)\in X_2$. Suppose $m$ is such a morphism. 
    \begin{align*}
        m(a,b)&=m(ec(b)-ec(b)+a, b)\\
        &=m(ec(b),b)+m(-ed(a)+a,0)\\
        &=b+m(-ed(a)+a, -ed(a)+ed(a))\\
        &=b-m(ed(a),ed(a))+m(a,ed(a))\\
        &=b-ed(a)+a\\
        &=b-ec(b)+a.
    \end{align*}
    Therefore we have uniqueness.  
\end{proof}

\begin{Remark}
The same results can also be viewed as a consequence of a result from M. Gran and D. Bourn in \cite{CatAspOfModula}, proposition 4.2 and 4.3. Even though the article from M.Gran and D.Bourn focuses on modularity of lattices of all equivalence relations over some object, uniqueness of the composition still holds in general if only lattices of kernel pairs over a given object are all modular (see the proofs in \cite{CatAspOfModula}).   
\end{Remark}

We can have a more precise idea of what $m$ looks like in general. Consider the fact that the reflexive graph $\RR$ is isomorphic to a reflexive graph of the form  

 \begin{center}
    \begin{tikzcd}
        (K\rtimes_\mu X_0, P) \arrow[-regepi, r, shift left=1.5ex, "p_1"]
      \arrow[-regepi, r, shift right=1.5ex, swap, "\partial p_0+p_1"] \arrow[r, <-<, swap, "i_1" description]& (X_0, P_{0})
   \end{tikzcd} 
   \end{center}

for some precrossed module of groups $(X_0,K,\mu,\partial)$(see section \ref{PXXModGrp}). The positive cone $P$ is simply the direct correspondence of $P_1$ under the group isomorphism $X_1\cong K\rtimes_\mu X_0$. Notice we do not ask $\mu$ or $\partial$ to respect in any way the preorder of the groups. Now the morphism $m$ can be written as $m((a,\partial(b)+x),(b,x))=(a+b,x)$. It is a well known result that $m$ is a group morphism exactly when the precrossed module $(X_0,K,\mu,\partial)$ is actually a crossed module, i.e. the Peiffer identity $\mu(\partial(k))(l)=klk\mm$, for all $k,l\in K$ holds (see section \ref{PXXModGrp}). Even then, it is not necessarily monotone, the monotonicity of $m$ is a property of the positive cone $P$ that will be further investigated in the context of $\textsf{S}$-reflexive graphs, in section \ref{Schreier}. 

Another remark is that in the category of groups, every internal category is a groupoid, as it is the case in any Mal'tsev category. This fails to be the case in the category $\crpogrp$. We will see in section \ref{examples} a variety of different internal categories  in $\crpogrp$ that do not necessarily have such good properties. Once again we are able to explicit the necessarily unique structure of groupoid over an internal category. Suppose we have an internal groupoid, i.e. an internal category together with a morphism $\sigma$

\begin{center}
        \begin{tikzcd}
    \CC=&|[alias=HL]|X_2&|[alias=HM]|X_1 &|[alias=HR]|X_0\\
    \arrow[->, from=HL, to=HM, "m"]  
    \arrow[->>, from=HM, to=HR, shift left=4, "d"]
    \arrow[->>, from=HM, to=HR, shift right=4, "c"']
    \arrow[>->, from=HR, to=HM, shift left=0, "e" description]
    \arrow[loop above, from=HM, "\sigma"]
        \end{tikzcd}
\end{center}
such that $m(a, \sigma(a))=ec(a)$ and $m(\sigma(a),a)=ed(a)$. Then, for any $a\in X_1$ as in the diagram above,

\begin{equation}\label{uniquenesss}    
\begin{split}
&m(\sigma(a),a)=ed(a)\\
    \lr&a-ec(a)+\sigma(a)=ed(a)\\ 
    \lr&\sigma(a)=ec(a)-a+ed(a).
\end{split}
\end{equation}

Therefore we have an explicit expression of the morphism $\sigma$. Moreover, as $X_1\cong \Ker{d}\rtimes_\mu X_0$, for some unspecified positive cone of the group $\Ker{d}\rtimes_\mu X_0$ and some fitting $\mu$, as seen above, 

\begin{align*}
    \sigma(a,b)&=i_1(cp_0+p_1)(a,b)-(a,b)+i_1p_1(a,b)\\
    &=(0,c(a)+b)+(\mu(-b)(-a), -b)+(0,b)\\
    &=(0,c(a)+b)+(\mu(-b)(-a),0)\\
    &=(\mu(c(a))(-a), c(a)+b)\\
    &=(-a, c(a)+b).
\end{align*}

This map is always a group morphism (as a consequence once again of $\cgrp$ being a Mal'tsev category). However, this map $\sigma$ fails to be monotone in general (see section \ref{Schreier}). Let us characterize an especially important class of groupoids, the class of effective equivalence relations. 

\begin{Proposition}\label{InternalEffective}
    Let $(G,P_G)$ be a right-preordered group. For any $b\in G$, define $P_{G}+b=\{x\in G\vert x-b\in P_G\}$. Then an internal equivalence relation (i.e. as below, with $G$ a normal subgroup of $X_0$, and $\mu$ the action via conjugation)  
    
    \begin{center}
        \begin{tikzcd}
	    |[alias=HM]|(X_1=G\rtimes_\mu X_0, P_1) &|[alias=HR]|(X_0, P_0)\\  
    \arrow[-regepi, from=HM, to=HR, shift left=4, "p_1"]
    \arrow[-regepi, from=HM, to=HR, shift right=4, "p_0+p_1"']
    \arrow[>->, from=HR, to=HM, shift left=0, "i_1" description]
    \arrow[loop above, from=HM, "\sigma"]
        \end{tikzcd}
    \end{center}

    is effective if and only if $P_1=\bigcup_{b\in P_0}((P_0-b)\cap G)\times \{b\}$.
\end{Proposition}
\begin{proof}
    The coequalizer of the arrows $(p_1, p_0+p_1)$ is the arrow $(f,\overline{f}=f_{\vert P_0})$ with codomain the group $X_0/G$ with positive cone the direct image $f(P_0)$ of $P_0$ under the quotient of $X_0$ by $G$. Because limits are computed componentwise in $\crpogrp$, the positive cone of a kernel pair of $(f,\overline{f})$ is the kernel pair of $\overline{f}$ in the category of monoids. Suppose $((X_1,P_1), p_1, p_0+p_1, i_0)$ is an effective equivalence relation over $(X_0, P_0)$ . We have the equivalences 
    \begin{align*}
         &(a,b)\in P_1\\
         \lr &f(b)=f(a+b)\text{ and } b\in P_0, a+b\in P_0\\
         \lr &f(a)=0  \text{ and } b\in P_0, a\in P_0-b\\
         \lr & a\in (P_0-b)\cap G, b\in P_0.
    \end{align*}
    The fact that $P_1$ is the positive cone of an effective equivalence relation implies that $P_1=\bigcup_{b\in P_0}((P_0-b)\cap G)\times \{b\}$. This equivalence relation is the kernel pair of its coequalizer $((f,\overline{f}), (X_0/G, f(P_0)))$. This concludes the proof. 
\end{proof}

\section{\textsf{S}-crossed modules and \textsf{S}-action representability}\label{Schreier}

\subsection{Schreier split epimorphisms and \textsf{S}-semi-direct product}

The category $\crpogrp$ is not protomodular, as shown in \cite{RPOgrpCatPersp}. However, this is also the case for the category $\cmon$, but it has been shown that restricting our interest to a special class of split epimorphisms in the category $\cmon$, we can still get some weaker notion of protomodularity (see \cite{MonoidsandpointedSprotomodularcategories}). This class of split epimorphisms is the class of Schreier split epimorphisms. We recall the definition. 

\begin{Definition}[Schreier split epimorphisms in $\cmon$, \cite{Semidirectproductsandcrossedmodulesinmonoidswithoperations}]
    Let 
   \begin{tikzcd}
      K\arrow[>->, r, shift left=0.5ex]\arrow[r, <-, shift right=0.5ex, swap, dotted, "s"]& M \arrow[->>, r, shift left=0.5ex, "d"]
      \arrow[r, <-<, shift right=0.5ex, swap, "e"]& N
   \end{tikzcd} 
   be a split epimorphism in $\cmon$, where $K$ is the kernel of $d$. If there exists a unique set theoretic map $s$ (depicted as the dotted arrow above) such that for all $m\in M$, $m=s(m)+ed(m)$, we say that the split extension is a Schreier split extension, and the split epimorphism is said to be a Schreier split epimorphism. 
\end{Definition}

We adapt this definition to the context of right-preordered groups as the following :

\begin{Definition}[Schreier split epimorphism in $\crpogrp$]
   Let 
   \begin{tikzcd}
       (X_1, P_1) \arrow[->, >=regepi, r, shift left=0.5ex, "d"]
      \arrow[r, <-<, shift right=0.5ex, swap, "e"]& (X_0, P_0)
   \end{tikzcd} be a split epimorphism in $\crpogrp$. We will call such a split epimorphism a Schreier split epimorphism if the induced split extension \begin{tikzcd}
       P_1 \arrow[->>, r, shift left=0.5ex, "\overline{d}"]
      \arrow[r, <-<, shift right=0.5ex, swap, "\overline{e}"]& P_0
   \end{tikzcd} over positive cones is a Schreier split epimorphism in $\cmon$.

   Since our monoids are embedded in groups, then necessarily the map $s$ above satisfying $a=s(a)+ed(a)$ has to be $s(a)=a-ed(a)$. $s$ has to take values in $P_1$, thus we must ask that for all $a$ in $P_1$, $a-ed(a)$ is still in $P_1$. Such a map $s$ automatically takes value in the kernel of $\overline{d}$. Therefore, such a map $s$ exists if and only if   
   \begin{align*}
       \forall a \in P_1, a-ed(a)\in P_1.
   \end{align*}
   We shall write $\textsf{S}$ for this class of split epimorphisms. 
\end{Definition}

The main result about the class of Schreier split epimorphisms in $\cmon$ is that it allows us to define a suitable notion of protomodularity, relatively to this class of split epimorphisms (see \cite{Schreiersplitepimorphismsinmonoidsandinsemirings}, definition 8.1.1). We recall the definition of $S$-protomodularity in the context of a pointed category.

\begin{Definition}[$S$-protomodularity, \cite{MonoidsandpointedSprotomodularcategories}]
    Let $\CC$ be a finitely complete pointed category. Let $S$ be a pullback-stable class of split epimorphisms in $\CC$. We say that $\CC$ is $S$-protomodular if for any split epimorphism  
    \[\begin{tikzcd}
	K && X && A
	\arrow["k", monoreg->, from=1-1, to=1-3]
	\arrow["f", shift left=2, -regepi, from=1-3, to=1-5]
	\arrow["s", shift left=2, tail, from=1-5, to=1-3]
\end{tikzcd}\]
in $S$, $k$ and $s$ are jointly strongly epimorphic, and $S$ is closed under finite limits in $\textsf{SplitEpi}_\CC$ the category of split epimorphisms in $\CC$.
\end{Definition}

The main example of $S$-protomodular category is the category $\cmon$, for $S$ the class of Schreier split epimorphisms (see \cite{Schreiersplitepimorphismsinmonoidsandinsemirings}). We show that for the category $\crpogrp$ is also an $\textsf{S}$-protomodular category.

\begin{Theorem}
    The category $\crpogrp$ is $\textsf{S}$-protomodular.
\end{Theorem}
\begin{proof}
    We already know that $\crpogrp$ is pointed and finitely complete. $\textsf{S}$ is also pullback-stable since limits are computed componentwise in $\crpogrp$, the pullback of a split extension in $\crpogrp$ is the pullback in $\cgrp$, whose restriction to positive cones is the pullback in $\cmon$. By pullback-stability of the class of Schreier split epimorphisms in $\cmon$, $\textsf{S}$ is also pullback-stable. Similarly, since the class of Schreier split epimorphisms of monoids is closed under limits in $\textsf{SplEpi}_\cmon$, the class of Schreier split epimorphisms contains the terminal split extension $0\leftrightarrows 0$ in $\cmon$, therefore $\textsf{S}$ also contains the terminal split extension in $\crpogrp$. $\textsf{S}$ is also closed under equalizers since the class of Schreier split epimorphisms is in $\cmon$ and limits are computed componentwise, and closed under products for the same reason. Consider a split epimorphism 
    \[\begin{tikzcd}
	(K,P_K) && (X,P_X) && (A,P_A)
	\arrow["k", monoreg->, from=1-1, to=1-3]
	\arrow["f", shift left=2, -regepi, from=1-3, to=1-5]
	\arrow["s", shift left=2, tail, from=1-5, to=1-3]
\end{tikzcd}\]

in $\textsf{S}$ such that both $k$ and $s$ factor through a subobject $(S,P_S)$ of $(X,P_X)$. We then have $S\cong X$ in $\cgrp$ because $\cgrp$ is a protomodular category and $P_S\cong P_X$ in $\cmon$ because in $\cmon$ $k$ and $s$ are jointly extremally epimorphic when $(f,s)$ is a Schreier split epimorphism in $\cmon$ (see \cite{MonoidsandpointedSprotomodularcategories}). In a finitely complete category, a pair of morphisms with same codomain is jointly strongly epimorphic if and only if it is jointly extremally epimorphic (see appendix in \cite{Malcevprotomodularhomologicalandsemiabeliancategories}), therefore the proof is complete.
\end{proof}

\subsection{\textsf{S}-semi-direct product in $\crpogrp$}\label{SemiDirect}

Let us describe Schreier split extensions in the category $\crpogrp$ for some fixed kernel and codomain. In view of this goal, we define, given a right-preordered group $(G,P_G)$, $$\adacc{G,P_G}_\textsf{S}:=(\Aut{G},\Aut{G}_\leq).\label{Autrpogrp}$$ $\Aut{G}$ here represents the group of group automorphisms, i.e. automorphisms that respect the group structure but need not preserve the preorder, and $\Aut{G}_\leq$ is the submonoid of group automorphisms that preserve the preorder (note that for $f\in \Aut{G}_\leq$, we do not require $f\mm$ to preserve the preorder).

Let us consider a morphism $\mu: (X_0, P_0)\longrightarrow \adacc{X_1,P_1}_\textsf{S}$. This morphism defines a semi-direct product $(X_1, P_{X_1})\rtimes_\mu(X_0, P_{X_0})$ over the set $X_1\times X_0$ such that the second projection and both injections are morphisms in $\crpogrp.$ Equivalently, the group structure of the semi-direct product of group $X_1\rtimes_\mu X_0$, becomes a right-preordered group with the positive cone $P_1\times P_0$.  This is the case because for all $(a,x), (b,y)$ in $P_1\times P_0, (a,x)+(b,y)=(a+\mu(x)(b), x+y)\in P_1\times P_0 $, since $\mu(x)$ is a morphism in $\crpogrp$ as long as $x\in P_0$. This definition of semi-direct product of right-preordered groups is motivated by the following  : 

\begin{center}
\begin{tikzcd}
    |[alias=L]|(X_1, P_1)&|[alias=M]|(X_1, P_1)\rtimes_\mu (X_0, P_0)&|[alias=R]|(X_0, P_0)\\
    \ar[monoreg->,  from=L, to=M, "i_0"] \ar[>->,  from=R, to=M, "i_1", shift left=1.5ex] \ar[-regepi,  from=M, to=R, "p_1", shift left=1.5ex]
\end{tikzcd}    
\end{center}

is a Schreier split extension in the category $\crpogrp$, indeed, for all $(a,b)\in P_1\times P_0$, $(a,b)-(0,b)=(a,0)\in P_1\times P_0$.  Moreover, the converse also holds, and this is the content of the following lemma. 

\begin{Lemma}\label{Actrepresentability}
	Any Schreier split extension in $\crpogrp$ gives rise to some semi-direct product as described above.
\end{Lemma}
\begin{proof}
 Consider a Schreier split extension in $\crpogrp$.  

\begin{center}
\begin{tikzcd}
    |[alias=L]|(K,P_K)&|[alias=M]|(G, P_G)&|[alias=R]|(X_0, P_0).\\
    \ar[monoreg->,  from=L, to=M, "i"] \ar[>->,  from=R, to=M, "j", shift left=1.5ex] \ar[-regepi,  from=M, to=R, "p", shift left=1.5ex]
\end{tikzcd}    
\end{center}

As groups (not right-preordered groups), $G\cong K\rtimes_\mu X_0$, for some group action via automorphisms of groups $\mu$ (see section \ref{PXXModGrp}). The split extension in $\crpogrp$ is then isomorphic to 

\begin{center}
\begin{tikzcd}
    |[alias=L]|(K,P_K)&|[alias=M]|(K\rtimes_\mu X_0, P)&|[alias=R]|(X_0, P_0).\\
    \ar[monoreg->,  from=L, to=M, "i_0"] \ar[>->, from=R, to=M, "i_1", shift left=1.5ex] \ar[-regepi,  from=M, to=R, "p_1", shift left=1.5ex]
\end{tikzcd}    
\end{center}

The positive cone $P$, which is simply the direct image of $P_G$ under the group isomorphism $G\cong K\rtimes_\mu X_0$, also has to satisfy $P_1\times P_0=P_{prod}\subseteq P\subseteq P_{lex}=(P_1\times P_0^0)\cup (X_1\times (P_0\backslash P_0^0))$ (with $P_0^0:=\{a\in P_0 \vert -a\in P_0\}$) (\cite{RPOgrpCatPersp}, proposition 4.1). 

Under this considerations, being a Schreier split extension in the category of right-preordered groups is exactly the condition
$$
\forall (a,b)\in P, (a,0)\in P.
$$

This condition can only be satisfied if $P=P_{prod}$, and we fall back to the case of semi-direct products as described above. 
\end{proof}

We shall now call such a semi-direct product an $\textsf{S}\mhyphen$semi-direct product of right preordered groups.  

\subsection{\textsf{S}-action representability and \textsf{S}-centre}

Schreier split extensions in the category $\cmon$ of monoids satisfy stronger categorical properties, reminiscent of the properties satisfied by all split extensions in $\cgrp$. In order to formally express this, we give the following definition. 

\begin{Definition}[$S$-action representable category, \cite{AlgebraicexponentiationandactionrepresentabilityforVgroups}]
    Let $\CC$ be a pointed $S$-protomodular category. We say that $\CC$ is $S$-action representable if, for any object $X$ in $\CC$, the functor $S\mhyphen\cPt{\_,X}:\CC\rightarrow \cset$ sending an object $Y\in \CC$ to the set of isomorphisms classes of $S$-split extensions with codomain $Y$ and kernel $X$ is representable. 
\end{Definition}

We can give some examples of such categories :

\begin{itemize}
    \item The category $\cgrp$ if $S$ is the class of all split epimorphisms, then $S\mhyphen\cPt{Y,X}\cong \cgrp(Y, \Aut{X})$;
    \item The category $\cmon$ if $S$ is the class of Schreier split epimorphisms, $S\mhyphen\cPt{Y,X}\cong \cmon(Y, \End{X})$ (\cite{Schreiersplitepimorphismsinmonoidsandinsemirings}, section 5.);
    \item An entire class of examples closely related to right-preordered groups, the categories of $V\mhyphen$groups, for $V$ a cartesian quantale, has been investigated in \cite{AlgebraicexponentiationandactionrepresentabilityforVgroups} by M.M. Clementino and A. Montoli, where $S$ is the class of split epimorphisms whose domain has the product $V$-structure.
\end{itemize}

In our context, the assignement that will play the role of $\Aut{\_}$ for $\cgrp$ or $\End{\_}$ for $\cmon$ will be $\adacc{\_}_{\textsf{S}}$. We have 

\begin{Theorem}\label{SActRepr}
    In $\crpogrp$, $$\textsf{S}\mhyphen\cPt{\_,(X,P_X)}\cong \crpogrp(\_,\adacc{X,P_X}_{\textsf{S}}).$$
\end{Theorem}
\begin{proof}
    Let $(X,P_X)$ be an object in $\crpogrp$. We will show that for any $\textsf{S}$-split short exact sequence $A$ with kernel $(X,P_X)$, exists a unique morphism of $\textsf{S}$-split short exact sequence from $A$ to the $\textsf{S}$-split short exact sequence 
    \begin{center}
    \begin{tikzcd}
        (X,P_X)\ar[monoreg->,r, "i_0"]&(X,P_X)\rtimes_{\id{{\adacc{X,P_X}_{\textsf{S}}}}} \adacc{X,P_X}_\textsf{S}\ar[r, -regepi, shift left=1, "p_1"]&\adacc{X,P_X}_\textsf{S}\ar[l, >->, shift left=1, "i_1"]
    \end{tikzcd}
    \end{center}    
where the object in the middle is the semi-direct product from section \ref{SemiDirect}. 

Let us consider an object $(Y,P_Y)$ in $\crpogrp$. We just proved in section \ref{SemiDirect} that any $\textsf{S}$-split short exact sequence of $\textsf{S}\mhyphen \cPt{(Y,P_Y),(X,P_X)}$ has as domain an $\textsf{S}\mhyphen$semi-direct product of $(Y,P_Y)$ and $(X,P_X)$, and the converse is lemma \ref{Actrepresentability}. We just need to prove the uniqueness. Suppose we have two morphisms $\mu,\nu$ in $\crpogrp((Y,P_Y),\adacc{X,P_X}_\textsf{S})$ such that the induced $\textsf{S}$-split short exact sequences are isomorphic. In particular, the semi-direct products $(X,P_X)\rtimes_\mu (Y,P_Y)$ and $(X,P_X)\rtimes_\nu (Y,P_Y)$ are isomorphic. Let us write the isomorphism $\psi : (X,P_X)\rtimes_\mu (Y,P_Y)\rightarrow (X,P_X)\rtimes_\nu (Y,P_Y)$. By definition of morphisms of short exact sequences, the following diagram has to commute   

\begin{center}
\begin{tikzcd}
      (X,P_X)\ar[monoreg->, r, "i_0"]\ar[d, equal]&(X,P_X)\rtimes_\mu (Y,P_Y)\ar[d, "\psi"]&(Y,P_Y)\ar[d, equal]\ar[l, >->, "i_1"]\\
      (X,P_X)\ar[monoreg->, r, "i_0"]&(X,P_X)\rtimes_\nu (Y,P_Y)&(Y,P_Y)\ar[l, >->, "i_1"].
\end{tikzcd}
\end{center}

Set-theoretically, $\psi$ then has to be the identity, and thus it is also the identity in $\crpogrp$. From there for all $x\in X, y\in Y$, $\psi((\mu(y)(x),0_Y))=(\mu(y)(x),0_Y)=(\nu(y)(x),0_Y)$, thus $\mu=\nu$. This concludes the proof, since the naturality of the isomorphism of functors follows from the uniqueness. 
\end{proof}

From the concept of action representability, one can define some categorical notion of centre of an object (see \cite{ActionAccessibilityViaCentralizers}), for example, this construction in $\cgrp$ applied to a group $G$ produces the usual centre of the group $Z(G)\subseteq G$ of elements in $G$ that commute with any other element of $G$. We adapt the definition to the context of $S\mhyphen$action representability.

\begin{Definition}[$S$-Centre]
    Let $\CC$ be an $S$-action representable category, $X$ an object in $\CC$ and suppose that the indiscrete split extension

    \begin{center}
\begin{tikzcd}
      X\ar[monoreg->, r, "i_0"]&\nabla_X \ar[-regepi, r, "p_1", shift left=1]&X\ar[l, >->, "\Delta_X", shift left=1]
\end{tikzcd}
\end{center}
belongs to $S$, with $\Delta_X$ the diagonal morphism. If $S\mhyphen \cPt{\_,X}\cong \CC(\_,\adacc{X}_S)$, then the $S$-centre of $X$ is the kernel of the morphism $\mu:X\rightarrow \adacc{X}_S$ that corresponds to the above indiscrete split epimorphism via the isomorphism. We write $Z_S(X)$ for this object of $\CC$. 
\end{Definition}

Now we apply this definition to $\crpogrp$ and the class $\textsf{S}$ of Schreier split epimorphisms. Note that the condition about the indiscrete split extension belonging to $S$ is not as trivial as it might seem. Indeed, here, such a centre could not be described with this definition for all objects. The indiscrete relation is always effective in a pointed category, as the kernel pair of the terminal map. Considering this, we can use the description of effective equivalence relations from \ref{InternalEffective}. Consider a right-preordered group $G$ with positive cone $P_G$. The indiscrete point is 

    \begin{center}
\begin{tikzcd}
      G\ar[>->, r, "i_0"]&G\rtimes_\mu G \ar[->>, r, "p_1", shift left=1]&G\ar[l, >->, "i_1", shift left=1]\\
      P_G\ar[>->, r, "\overline{i_0}"]\ar[u,>->]&\bigcup_{g\in P_G}(P_G-g)\times \{g\} \ar[u,>->]\ar[->>, r, "\overline{p_1}", shift left=1]&P_G\ar[u,>->]\ar[l, >->, "\overline{i_1}", shift left=1].
\end{tikzcd}
\end{center}

The semi-direct product is relative to the conjugation action of $G$ over itself.

This is an element of $\textsf{S}$ if and only if $\bigcup_{g\in P_G}(P_G-g)\times \{g\}$ is the product preorder, i.e. $P_G\times P_G$. This is equivalent to asking that for all $g\in P_G$, $P_G-g=P_G$, which means that for all $g\in P_G$, $x+g\in P_G$ if and only if $x\in P_G$. Now because $0=-g+g\in P_G$, we must have $-g\in P_G$, so $P_G$ has to be a group. Conversely if $P_G$ is a group, $x+g\in P_G$ with $g\in P_G$ implies $x\in P_G$, so $P_G-g=P_G$. The centre $Z_\textsf{S}((G,P_G))$ can only be computed when $P_G$ is a group, and then the morphism $\mu: g\mapsto (h\mapsto g+h-g)$ has as kernel the preordered abelian group $Z(G)$ (the usual group-theoretic centre of the group $G$), with positive cone $Z(P_G)$. One could observe that the objects in $\crpogrp$ for which there is a notion of centre are exactly the protomodular objects of $\crpogrp$ (see \cite{RPOgrpCatPersp}, theorem 3.4).

\subsection{Reflexive graphs and \textsf{S}-precrossed modules}\label{RG}

From the relative version of protomodularity in $\cmon$, the theory of crossed modules can actually be partially reconstructed (see \cite{Patchkoria}). We get a weaker notion of crossed module, that A. Patchkoria called crossed semi-module in \cite{Patchkoria}. In this section we will get some notion of $\textsf{S}\mhyphen$crossed module of right-preordered groups, based of the work of A. Patchkoria in \cite{Patchkoria}. To get a detailed understanding, we will first define a notion of relative precrossed modules, such that it forms a more algebraic category equivalent to the category of internal \textsf{S}-reflexive graphs. We begin with some definitions. Note that apart from restriction to the class $\textsf{S}$, every argument and construction will be similar to the case of groups, recalled in the introduction. 

\begin{Definition}[$\textsf{S}$-reflexive graph, \cite{MonoidsandpointedSprotomodularcategories}]
    A reflexive graph
\begin{center} 
\begin{tikzpicture}[descr/.style={fill=white},text height=1.5ex, text depth=0.25ex]
\node (b) at (2.5,0) {$(Y_1, P_{Y_1})$};
\node (c) at (5,0) {$(Y_0, P_{Y_0})$};
\path[->,font=\scriptsize]
([yshift= 9pt]b.east) edge[->, >=regepi] node[above] {$d$} ([yshift= 9pt]c.west)
(c.west) edge[>->] node[descr] {$e$} (b.east)
([yshift= -9pt]b.east) edge[->, >=regepi] node[below] {$c$} ([yshift= -9pt]c.west);
\end{tikzpicture}
\end{center}

in $\crpogrp$, will be said to be an $\textsf{S}$-reflexive graph if 

\begin{center} 
\begin{tikzpicture}[descr/.style={fill=white},text height=1.5ex, text depth=0.25ex]
\node (b) at (2.5,0) {$(Y_1, P_{Y_1})$};
\node (c) at (5,0) {$(Y_0, P_{Y_0})$};
\path[->,font=\scriptsize]
([yshift= 5pt]b.east) edge[->, >=regepi] node[above] {$d$} ([yshift= 5pt]c.west)
([yshift= -5pt]c.west) edge[>->] node[below] {$e$} ([yshift= -5pt]b.east);
\end{tikzpicture}
\end{center}
   
   is a Schreier split epimorphism in $\crpogrp$. These reflexive graphs together with the usual morphisms define a category $\cSchRG{\crpogrp}$. We will write $((Y_0, P_{Y_0}), (Y_1, P_{Y_1}), d,c,e)$ to denote an object in $\cSchRG{\crpogrp}$. 
\end{Definition}

\begin{Definition}[$\textsf{S}$-precrossed module in $\crpogrp$]
    An $\textsf{S}\mhyphen$precrossed module of right-preordered groups is the data of two right-preordered groups $X_0$ and $X_1$, with positive cones $P_0$ and $P_1$ respectively, together with monotone group morphisms $\partial:(X_1, P_1)\longrightarrow (X_0, P_0)$ and $\mu: (X_0, P_0)\longrightarrow \adacc{X_1,P_1}_\textsf{S}$. The following axiom must be satisfied 
    \begin{itemize}
        \item[\textbf{(PX)}] $\forall a\in X_1, \forall x\in X_0, \partial(\mu(x)(a))=x+\partial(a)-x.$
    \end{itemize}

    Morphisms between $\textsf{S}\mhyphen$precrossed modules of right-preordered groups $((X_0, P_0),(X_1, P_1),\partial,\mu)$ and $((X_0', P_0'),$ $(X_1', P_1'),\partial', \mu')$ are pairs $(\gamma_0 : (X_0,P_0)\to (X_0',P_0'), \gamma_1 : (X_1,P_1)\to (X_1',P_1'))$  of right-preordered group morphisms such that 
    \begin{itemize}
        \item $\gamma_1\partial=\partial'\gamma_0$,
        \item $\forall x\in X_0, \forall a\in X_1, \gamma_1(\mu(x)(a))=\mu'(\gamma_0(x))(\gamma_1(a)).$
    \end{itemize}

    The notion of $\textsf{S}\mhyphen$precrossed module of right-preordered groups together with morphisms of $\textsf{S}\mhyphen$precrossed modules of right-preordered groups defines a category $\cpxsmod{\crpogrp}$.
\end{Definition}

These two categories will be shown to be equivalent. To prepare this result, we define the functor $\fF$ from $\cSchRG{\crpogrp}$ to $\cpxsmod{\crpogrp}$ which associates with each reflexive graph 

\begin{center}
        \begin{tikzcd}
	   \RR=&|[alias=HM]|(X_1,P_1) &|[alias=HR]|(X_0, P_0)\\
    \arrow[-regepi, from=HM, to=HR, shift left=4, "d"]
    \arrow[-regepi, from=HM, to=HR, shift right=4, "c"']
    \arrow[>->, from=HR, to=HM, shift left=0, "e"]
        \end{tikzcd}
    \end{center}

the precrossed module $((\Ker{d}, \Ker{d}\cap P_1),(X_0, P_0),  \mu_\RR , \partial_\RR)$, where for any $x$ in $\Ker{d}, y$ in $X_0$, $\partial_\RR(x)=c(x)$ and $ \mu_\RR(y)(x)=e(y)+x-e(y)$. With any morphism $(f_0,f_1)$ of internal reflexive graphs in $\crpogrp$ we simply associate the morphism $(f_0, f_1')$, with $f_1'$ the restriction of $f_1$ to $\Ker{d}$. We check that this functor is well-defined. 

For an object $\RR=((X_0,P_0),(X_1,P_1),d,c,e)$ in $\cSchRG{\crpogrp}$, she shall show that $\fF(\RR)$ is an object in $\cpxsmod{\crpogrp}$. The main point is that $\mu_\RR$ is a morphism $(X_0,P_0)\rightarrow \adacc{\Ker{d}, \Ker{d}\cap P_1}_\textsf{S}$ in $\crpogrp$, as in section \ref{Autrpogrp}. It is a group morphism that satisfies the identity \textbf{(PX)}, indeed for any $x$ in $X_0$, $a$ in $\Ker{d}, c(\mu_\RR(x)(a))=c(e(x)+a-e(x))=x+c(a)-x$. Moreover, if $x$ in $P_0, \mu_\RR(x)$ preserves the positive cone $\Ker{d}\cap P_1$, indeed, if $a$ in $\Ker{d}\cap P_1$, 
\begin{align*}
            \mu_\RR(x)(a)&=e(x)+a-e(x)\\
            &=e(x)+a-ed(e(x)+a)\in P_1\cap \Ker{d}.
\end{align*}
We used the fact that the underlying split epimorphism is a Schreier split epimorphism.  

Now let $(f_0,f_1)$ be a morphism in $\cSchRG{\crpogrp}$, with domain $\RR=((X_0,P_0),(X_1,P_1),d,c,e)$ and codomain $\RR'=((X_0,P_0)', (X_1,P_1)', d',c',e')$. The codomain of $f_1'$ is contained in $\Ker{d'}$ : if $a\in \Ker{d}, d'(f_1(a))=f_0(d(a))=0$ therefore $f_1'(a)\in \Ker{d'}$.  We also get the compatibility rule for morphisms of $\textsf{S}\mhyphen$precrossed module of right-preordered groups, $f_0c_{\vert \Ker{d}}=c'_{\vert \Ker{d'}}f_1$ and $f_1\mu_\RR(x)(a)=\conj{f_1e(x)}(f_1(a))=\conj{e f_0(x)}(f_1(a))=\mu_\RR(f_0(x))(f_1(a))$ because these are compatibility rules for morphisms of reflexive graphs.

$\fF$ is therefore well defined as a mapping, and functoriality is immediate. We have the two following lemmas.

\begin{Lemma}
$\fF$ is essentially surjective on objects.
\end{Lemma}
\begin{Proof}
Given an $\textsf{S}\mhyphen$precrossed module of right-preordered groups $\XX=((K, P_K), (X, P_X), \mu, \partial)$, we just have to check that the reflexive graph 
\begin{center}
        \begin{tikzcd}
\RR=&|[alias=HM]|(K,P_K)\rtimes_\mu (X,P_X) &|[alias=HR]|(X_0, P_0)\\
    \arrow[-regepi, from=HM, to=HR, shift left=4, "p_1"]
    \arrow[-regepi, from=HM, to=HR, shift right=4, "\partial p_0+p_1"']
    \arrow[>->, from=HR, to=HM, shift left=0, "i_1"]
        \end{tikzcd}
    \end{center}

is an $\textsf{S}$-reflexive graph of right-preordered groups. Since $\partial$ is a morphism of right-preordered groups and the positive cone of the $\textsf{S}\mhyphen$semi-direct product is endowed with the product positive cones $P_K\times P_X$, $\partial p_0+p_1$ is a morphism of right-preordered groups. We already proved in \ref{SemiDirect} that $\RR$ is an $\textsf{S}$-reflexive graph. It is immediate to check that $\fF(\RR)\cong \XX$.
\end{Proof}

\begin{Lemma}
$\fF$ is a fully faithful functor.
\end{Lemma}
\begin{Proof}
We proved in \ref{SemiDirect} that any Schreier split epimorphism in $\crpogrp$ is (up to isomorphism) a Schreier split epimorphism that has an $\textsf{S}$-semi-direct product of right-preordered groups as domain. Moreover the $\textsf{S}\mhyphen$semi-direct product involves only the objects $(X_0, P_0)$ and $(\Ker{d},\Ker{d}\cap P_1)$ (if we consider the same notations as above), thus $\fF$ is necessarily faithful. It is also full; to see this, we consider two $\textsf{S}$-reflexive graphs $\RR=((X_0,P_0), (X_1,P_1), d,c,e)$, $\RR'=((X_0',P_0'), (X_1',P_1'), d',c',e')$ in $\crpogrp$ and their respective image through $\fF$. Consider any morphism in $\cpxsmod{\crpogrp}(\fF(\RR), \fF(\RR'))$, say $(g_0, g_1)$. We know that there exists an isomorphism of right-preordered groups $\psi: (X_1,P_1)\rightarrow (\Ker{d},\Ker{d}\cap P_1)\rtimes_{\mu_\RR}(X_0,P_0)$ that extends to some isomorphism of $\textsf{S}\mhyphen$short exact sequences (\ref{SemiDirect}), giving it the structure of $\textsf{S}\mhyphen$reflexive graph by adding a morphism $\gamma=c\psi\mm : (\Ker{d},\Ker{d}\cap P_1)\rtimes_{\mu_\RR}(X_0,P_0)\rightarrow (X_0,P_0)$, we get an isomorphism of $\textsf{S}$-reflexive graph, as follows 

\begin{center}
        \begin{tikzcd}
|[alias=HM]|(X_1,P_1) \ar[d, "\psi"']&|[alias=HR]|(X_0,P_0)\ar[d, equal]\\    
|[alias=HMM]|(\Ker{d},\Ker{d}\cap P_1)\rtimes_{\mu_\RR}(X_0,P_0) &|[alias=HRR]|(X_0,P_0)\\
\arrow[->>, from=HM, to=HR, shift right=4, "c" description]
    \arrow[->>, from=HM, to=HR, shift left=4, "d"]
    \arrow[>->, from=HR, to=HM, shift left=0, "e" description]
    \arrow[->>, from=HMM, to=HRR, shift left=4, "p_1" description]
    \arrow[->>, from=HMM, to=HRR, shift right=4, "c\psi\mm"']
    \arrow[>->, from=HRR, to=HMM, shift left=0, "i_1" description]
        \end{tikzcd}
\end{center}

Doing the same for $\RR'$, we get another isomorphism of $\textsf{S}$-reflexive graphs $\psi'$. 
We will now just have to prove that $((\psi')\mm(g_1\times g_0)\psi, g_0)$ is a morphism of $\textsf{S}$-reflexive graphs, and then we would immediately get $(g_0,g_1)=\fF((\psi')\mm(g_1\times g_0)\psi, g_0)$, thus fullness of $\fF$. Since we know $\psi,\psi'$ are isomorphisms in $\cSchRG{\crpogrp}$, it is enough to check that the following diagram is a morphism in $\cSchRG{\crpogrp}$. 
\begin{center}
        \begin{tikzcd}
|[alias=HM]| (\Ker{d},\Ker{d}\cap P_1)\rtimes_{\mu_\RR}(X_0,P_0)\ar[dd, "g_1\times g_0"']&|[alias=HR]|(X_0,P_0)\ar[dd, "g_0"]\\ 
&&\\
|[alias=HMM]|(\Ker{d'},\Ker{d'}\cap P_1')\rtimes_{\mu_{\RR'}}(X_0',P_0') &|[alias=HRR]|(X_0', P_0').\\
\arrow[-regepi, from=HM, to=HR, shift right=4, "c\psi\mm"']
    \arrow[-regepi, from=HM, to=HR, shift left=4, "p_1"]
    \arrow[>->, from=HR, to=HM, shift left=0, "i_1" description]
    \arrow[-regepi, from=HMM, to=HRR, shift left=4, "p_1" description]
    \arrow[-regepi, from=HMM, to=HRR, shift right=4, "c'(\psi')\mm"']
    \arrow[>->, from=HRR, to=HMM, shift left=0, "i_1" description]
        \end{tikzcd}
\end{center}

Referring back to diagram (\ref{equivalencemorphismsRG}), this diagram is a morphism of reflexive graphs of groups. Because the domain of this reflexive graph of groups has as positive cone the direct product and $g_1,g_0$ are morphisms of right-preordered groups, the diagram above is a morphism of $\textsf{S}$-reflexive graphs of right-preordered groups, thus the result. 
\end{Proof}

We can now state the equivalence.

\begin{Theorem}
There is an equivalence of categories $$\cSchRG{\crpogrp}\cong\textsf{S}\mhyphen\cpxmod{\crpogrp}.$$

\end{Theorem}

Now that this theorem is established, we give the definition of $\textsf{S}$-crossed module in the category $\crpogrp$, and prove that a reflexive graph in $\crpogrp$ is an internal category if and only if its image through the functor $\fF$ is an $\textsf{S}\mhyphen$crossed module of right-preordered groups.

\begin{Definition}[$\textsf{S}\mhyphen$Crossed module in $\crpogrp$]
   An $\textsf{S}\mhyphen$crossed module is an $\textsf{S}\mhyphen$precrossed module $((X_1,P_1),(X_0,P_0),\mu,\partial)$ such that the following Peiffer axiom \textbf{(P)} is satisfied :
 \begin{align*}
     \textbf{(P)}\quad \forall a,b\in X_1, \mu(\partial(a))(b)=a+b-a.
 \end{align*}
  
  We write $\cxsmod{\crpogrp}$ for the full subcategory of $\cpxsmod{\crpogrp}$ where objects are $\textsf{S}\mhyphen$crossed modules.
\end{Definition}

\begin{Theorem}\label{EquivSchCatPeiffer}
    Let $\cSchCat{\crpogrp}$ denote the full subcategory of internal \textsf{S}-categories, that is, $\cSchRG{\crpogrp}$ whose objects are internal categories in $\crpogrp$. We have an equivalence of categories
    \begin{center}
        \begin{align*}
            \cxsmod{\crpogrp}\cong \cSchCat{\crpogrp}.
        \end{align*}
    \end{center}
\end{Theorem}
\begin{proof}
    We will prove that, under the previous equivalence $$\cpxsmod{\crpogrp}\cong\cSchRG{\crpogrp},$$ the two full subcategories $\cxsmod{\crpogrp}$ and $\cSchCat{\crpogrp}$ exactly correspond. 
    Suppose first $\CC=((X_1,P_1),(X_0,P_0),d,c,e,m)$ is an internal category, as in \ref{uniquenessm}. The Peiffer identity holds for the $\textsf{S}$-precrossed module $((\Ker{d},\Ker{d}\cap P_1),(X_0,P_0),\mu_\CC,\partial_\CC)$. Indeed, for all $a,b\in \Ker{d}$, since $m$ is a morphism, and $m(ec(a),a)=a=m(a,0)$ holds, 
    \begin{align*}
        \mu_\CC(\partial_\CC(a))(b)&=ec(a)+b-ec(a)\\
        &=m(ec(a)+b-ec(a),0)\\
        &=m(ec(a)+b-ec(a),a-a)\\
        &=m(ec(a),a)+m(b,0)-m(ec(a),a)\\
        &=a+b-a
    \end{align*}

    Conversely, if $((X_1,P_1),(X_0,P_0),\mu,\partial)$ satisfies $\textbf{(P)}$, then the unique $m$ defined in \ref{uniquenessm} is a mophism in $\crpogrp$. Indeed, with the Peiffer identity, we know that $m$ becomes a group morphism. Monotonicity comes from the expression $m((x,\partial(y)+a),(y,a))=(x+y,a)$, because the positive cone of $X_1\rtimes_\mu X_0$ is $P_{0}\times P_{1}$.  
\end{proof}

Lastly, we can characterize the internal \textsf{S}-categories that are actually groupoids. We get the following two corresponding categories :
\begin{itemize}
    \item $\cSchGrpd{\crpogrp}$ is the category whose objects are internal \textsf{S}-categories having a groupoid structure, written as $((X_0,P_0),(X_1,P_1),d,c,e,m,s)$. It is a full subcategory of $\cSchCat{\crpogrp}$.
    \item $\cxsmodp{\crpogrp}$ is the full subcategory of $\cxsmod{\crpogrp}$ such that for every object $((X_1,P_1),(X_0,P_0),\mu,\partial)$, $P_{1}$ is a group. 
\end{itemize}

We check that these categories correspond to each other via the equivalence. 

If $((X_0,P_0),(X_1,P_1),\mu,\partial)$ is an object in $\cxsmodp{\crpogrp}$, then the positive cone $P_{1}\times P_{0}$ is stable under the groupoid operation $s:(a,b)\mapsto (-a, \partial(a)+b)$ defined by equation (\ref{uniquenesss}). Conversely if such an operation $s$ is a morphism in $\crpogrp$, in some \textsf{S}-internal category $\CC=((X_0,P_0),(X_1,P_1),d,c,e,m)$, then because of the fact that $\CC$ is an \textsf{S}-internal category, we have an isomorphism with the internal category such that its \textquotedblleft object of morphisms" is the semi-direct product $(\Ker{d},\Ker{d}\cap P_1)\rtimes_{\mu_\CC} (X_0,P_0)$. Since $s$ is supposed monotone, for all $(a,x)\in (\Ker{d}\cap P_{1})\times P_{0}, (-a,c(a)+x)\in  (\Ker{d}\cap P_{1})\times P_{0}$. In particular, $-a\in \Ker{d}\cap P_{1}$ for all $a\in \Ker{d}\cap P_{1}$. 

 Now consider in an $\textsf{S}$-reflexive graph in $\crpogrp$

\begin{center}
        \begin{tikzcd}
    \RR=&|[alias=HM]|(K,P_K)\rtimes_\mu (X,P_X) &|[alias=HR]|(X,P_X),\\
    \arrow[-regepi, from=HM, to=HR, shift left=4, "p_1"]
    \arrow[-regepi, from=HM, to=HR, shift right=4, "\partial p_0+p_1"']
    \arrow[>->, from=HR, to=HM, shift left=0, "i_1"]
        \end{tikzcd}
\end{center}
the two subobjects $(K, P_K)=\ker(p_1)$ and $(L,P_L):=(\{(a,-\partial(a))\vert a\in K\}, L\cap (P_K\times P_X))=\ker(\partial p_0+p_1)$ of $K\rtimes_\mu X$. If $(K,P_K)$ and $(L, P_L)$ commute in the sense of Huq, then as groups $K$ and $L$ commute in the sense of Huq, and the arrow of the property is the addition. The addition is always monotone therefore $(K,P_K)$ and $(L,P_L)$ commute in the sense of Huq as right-preordered groups exactly when $K$ and $L$ commute in the sense of Huq as groups, which is exactly when the Peiffer identity holds over the group reflexive graph $\pym_0(\RR)$, which lastly is exactly when $\RR$ is an internal category in $\crpogrp$. This observation actually leads to a relative version of the Smith-is-Huq property. 

We can restrict our attention to $\textsf{S}$-effective equivalence relations (effective equivalence relations such that both underlying split epimorphisms belong to $\textsf{S}$) and check if the following $\textsf{S}$-Smith-is-Huq property holds. 

\begin{Definition}[$S$-Smith-is-Huq property, \cite{OntheSmithisHuqconditioninSprotomodularcategories}]
    Let $\CC$ be a pointed $S$-protomodular category, and $S$ a class of split epimorphisms in $\CC$ that is closed under composition and contains all product projections
\begin{center}
 \begin{tikzcd}
        X\times Y \ar[r, -regepi, shift left=2, "p_1"]&Y.\ar[shift left=2, l, "i_1", >->]
    \end{tikzcd}
\end{center}

We say that $\CC$ satisfies the \textit{$S$-Smith-is-Huq property} if any pair of $S$-effective equivalence relations commute in the sense of Smith exactly when their normalization commute in the sense of Huq.
\end{Definition}

Before wondering if $\crpogrp$ satisfies the $\textsf{S}\mhyphen$Smith-is-Huq property, we give a counterexample to insure $\crpogrp$ does not satisfy the Smith-is-Huq property for all equivalence relations.

\begin{Example}\label{NotSmithisHuq}

Consider the indiscrete relation $\nabla=(\Zz^2, \Nn^2)$ over the right-preordered group $(\Zz, \Nn)$. $\nabla$ does not centralize itself in the sense of Smith. Indeed, consider the diagram 

 \begin{center}
        \begin{tikzcd}
            (\Zz^3, \Nn^3) \ar[shift left=2, "p_{1\mathpunct{,}2}", r, -regepi]\ar[shift left=2, "p_{0\mathpunct{,}1}", d, -regepi]&\nabla\ar[shift left=2, "i_{1\mathpunct{,}2}", l, >->]\ar[d, -regepi, "p_0", shift left=2]\\
            \nabla\ar[shift left=2, "p_{1}", r, -regepi]\ar[u, >->, "i_{0\mathpunct{,}1}", shift left=2]&(\Zz, \Nn).\ar[l, >->, "\Delta", shift left=2]\ar[u, >->, "\Delta", shift left=2]
        \end{tikzcd}
    \end{center}

A connector $p$ satisfying $pi_{1,2}=p_1: \nabla\rightarrow (\Zz,\Nn)$ and $pi_{0,1}=p_0: \nabla\rightarrow (\Zz,\Nn)$ would have to satisfy the Mal'tsev axioms $p(x,y,y)=x$ and $p(x,x,y)=y$. Suppose such a connector exists in $\crpogrp$, then we would have, for any $y \in \Zz$,
\begin{align*}
p(0,y,0)&=p(0,y,y)-p(0,0,y)\\
&=0-y=-y.
\end{align*}

From there we can easily check that the uniquely determined connector $p$ we get cannot be monotone, as for example $(0,1,0)\in \Nn^3$ but $p(0,1,0)=-1\not\in\Nn$.

However, the normalization of $\nabla$ is $(\Zz, \Nn)$, the object itself, and one can check that it commutes in the sense of Huq with itself as a subobject of itself, which is made clearer with the following diagram :

 \begin{center}
        \begin{tikzcd}
            (\Zz, \Nn)\ar[>->, "i_0", r]\ar[equal, dr]&(\Zz^2,\Nn^2)\ar[d, "+" ]&(\Zz,\Nn) \ar[>->, "i_1"', l]\ar[equal, dl]\\
            &(\Zz, \Nn).&
        \end{tikzcd}
   \end{center}
The two triangles commute due to the fact that every monoid and group involved is commutative, turning $+$ into a morphism in $\crpogrp$. One could also refer to proposition 3.6. in \cite{CentrExtPOGrp} for a more general statement, that actually any right-preordered (and thus preordered) abelian group commutes with itself in the sense of Huq.

We now go back to the $\textsf{S}\mhyphen$Smith-is-Huq property. It was proved in proposition 2.3.2., \cite{Schreiersplitepimorphismsinmonoidsandinsemirings}, that in the category $\cmon$ Schreier split epimorphisms are closed under composition. It follows immediately that $\textsf{S}$ is closed under composition in $\crpogrp$. $\textsf{S}$ also contains all product projections, we can think of them as being $\textsf{S}\mhyphen$semi-direct products with trivial action. 

Now recall theorem 4.2 of \cite{OntheSmithisHuqconditioninSprotomodularcategories}, in a way convenient for us. 

\begin{Theorem}[\cite{OntheSmithisHuqconditioninSprotomodularcategories}]\label{SSmithHuqThm}
    Let $\CC$ be an $S$-protomodular category such that $S$ is closed under composition and every product projection is in $S$. Suppose moreover that any reflexive graph 
    
    \[\begin{tikzcd}
	{X_1} && {X_0}
	\arrow["d", shift left=3, -regepi, from=1-1, to=1-3]
	\arrow["c"', shift right=3, -regepi, from=1-1, to=1-3]
	\arrow["e"{description}, tail, from=1-3, to=1-1]
\end{tikzcd}\]

such that both split epimorphisms $(d,e)$, and $(c,e)$ are in $S$ and the two subobjects $\Ker{d},\Ker{c}$ of $X_1$ commute in the sense of Huq, is an internal category. Then $\CC$ satisfies the $S$-Smith-is-Huq property.
\end{Theorem}

\begin{Proposition}
    The category $\crpogrp$ satisfies the $\textsf{S}$-Smith-is-Huq property. 
\end{Proposition}
\begin{proof}
    We already proved that $\textsf{S}$ is closed under composition and contains all product projections. Now given a reflexive graph 
\[\begin{tikzcd}
	\RR : &{(X_1,P_1)} && {(X_0,P_0)}
	\arrow["d", shift left=3, -regepi, from=1-2, to=1-4]
	\arrow["c"', shift right=3, -regepi, from=1-2, to=1-4]
	\arrow["e"{description}, tail, from=1-4, to=1-2]
\end{tikzcd}\]

such that $(K,P_K):=(\Ker{d},\Ker{d}\cap P_1)$ and $(L,P_L):=(\Ker{c},\Ker{c}\cap P_1)$ commute in the sense of Huq. The morphism $\phi : (K\times L, P_K\times P_L)\to (X_1,P_1)$ in the diagram below 

\[\begin{tikzcd}
	{(K,P_K)} & {(K\times L, P_K\times P_L)}& {(L,P_L)} \\
	\\
	& {(X_1,P_1)}
	\arrow["{i_0}", tail, from=1-1, to=1-2]
	\arrow[tail, from=1-1, to=3-2]
	\arrow["\phi"{description}, from=1-2, to=3-2]
	\arrow["{i_1}"', tail, from=1-3, to=1-2]
	\arrow[tail, from=1-3, to=3-2]
\end{tikzcd}\]

has to be the group law $(k,l)\in K\times L\mapsto k+l$. The $\textsf{S}$-precrossed module $\fF(\RR)=((K,P_K),(X_0,P_0),\mu_\RR,\partial_\RR)$, as constructed in \ref{RG}, then satisfies the Peiffer identity \textbf{(P)}, indeed for $k,l\in K$, since $ec(k)-k\in L$,
\begin{align*}
    \mu_\RR(\partial_\RR(k))(l)+k-l-k&=ec(k)+l-ec(k)+k-l-k\\
    &=\phi(k,-k+ec(k))+\phi(l,-ec(k)+k)-l-k\\
    &=\phi(k+l,0)-l-k\\
    &=0.
\end{align*}

According to theorem \ref{EquivSchCatPeiffer}, $\RR$ is then an internal category, and by applying theorem \ref{SSmithHuqThm}, we get the $\textsf{S}$-Smith-is-Huq property for the category $\crpogrp$.
\end{proof}

\end{Example}

\section{The diversity of internal categories in $\crpogrp$}

We conclude the article by presenting some examples showing the independence  of various internal structures in the category of right-preordered groups.\label{examples}

\begin{Example}[Internal \textsf{S}-category that is not a groupoid]
Consider the internal category 
     \begin{center}
        \begin{tikzcd}
    |[alias=HL]|\Zz&|[alias=HM]|\Zz\times \Zz &|[alias=HR]|\Zz\\
    &&\\
    |[alias=LL]| \Nn &|[alias=LM]|\Nn\times \Zz&|[alias=LR]|\Zz\\ 
    \arrow[>->, from=LL, to=HL]
    \arrow[>->, from=HL, to=HM, "i_0"]
    \arrow[>->, from=LL, to=LM, "\overline{i_0}"]   
    \arrow[>->, from=LR, to=HR] 
    \arrow[>->, from=LM, to=HM]
    \arrow[->>, from=LM, to=LR, shift left=4, "\overline{p_1}"]
    \arrow[->>, from=LM, to=LR, shift right=5, "\overline{p_0+p_1}"']
    \arrow[->>, from=HM, to=HR, shift left=4, "p_1"]
    \arrow[->>, from=HM, to=HR, shift right=4, "p_0+p_1"']
    \arrow[>->, from=LR, to=LM, shift left=0, "\overline{i_1}" description]
    \arrow[>->, from=HR, to=HM, shift left=0, "i_1" description]
        \end{tikzcd}
    \end{center}
$\Nn$ is not a group so $\sigma$ is not a morphism in $\crpogrp$, but if $(a,b)\in \Nn\times \Zz, (a,b)-(0,b)=(a,0)\in \overline{i_0}(\Nn)$. 
\end{Example}

\begin{Example}[Groupoid that is not an internal \textsf{S}-category]
Consider the following internal category, where $C_2=\{\pm 1\}$, $\Qq^*_{|\cdot |\leq 1 }=\Qq^*\cap [-1,1]$, $\Qq^*_{|\cdot |< 1 }=\Qq^*\cap ]-1,1[$ :
     \begin{center}
        \begin{tikzcd}
    |[alias=HL]|\Rr^*&|[alias=HM]|\Rr^*\times \Rr^* &|[alias=HR]|\Rr^*\\
    &&\\
    |[alias=LL]| C_2 &|[alias=LM]| (C_2\times \Qq^*_{|\cdot |\leq 1 })\cup(\Qq^*\backslash C_2\times \Qq^*_{|\cdot |< 1 })&|[alias=LR]|\Qq^*_{|\cdot |\leq 1 }\\ 
    \arrow[>->, from=LL, to=HL]
    \arrow[>->, from=HL, to=HM, "i_0"]
    \arrow[>->, from=LL, to=LM, "\overline{i_0}"]   
    \arrow[>->, from=LR, to=HR] 
    \arrow[>->, from=LM, to=HM]
    \arrow[->>, from=LM, to=LR, shift left=4, "\overline{p_1}"]
    \arrow[->>, from=LM, to=LR, shift right=5, "\overline{p_1}"']
    \arrow[->>, from=HM, to=HR, shift left=4, "p_1"]
    \arrow[->>, from=HM, to=HR, shift right=4, "p_1"']
    \arrow[>->, from=LR, to=LM, shift left=0, "\overline{i_1}"description]
    \arrow[>->, from=HR, to=HM, shift left=0, "i_1"description]
        \end{tikzcd}
    \end{center}

The morphism $\sigma$ here is $(a,b)\mapsto (a\mm, b)$ and the positive cone is preserved under it. However, for example $(5,1/2) - (1, 1/2)=(5,1)\not\in \overline{i_0}(C_2).$ Notice we use the additive notation to be coherent with the rest of the article even though it represents the usual multiplication of rational numbers.
It is useful to check that the positive cone $P:=C_2\times \Qq^*_{|\cdot |\leq 1 }\cup \Qq^*\backslash C_2\times \Qq^*_{|\cdot |< 1 }$ is indeed a monoid. 
\begin{itemize}
    \item $P$ has a neutral element $(1,1)\in C_2\times\Qq^*_{|\cdot |\leq 1 }$,
    \item $(a,b)+(x,y)=(ax,by)$ and now we consider different possibilities :
    \begin{itemize}
        \item $(a,b),(x,y)\in C_2\times \Qq^*_{|\cdot |\leq 1 }\ri (ax,by)\in  C_2\times \Qq^*_{|\cdot |\leq 1 }$,
        \item $(a,b),(x,y)\in \Qq^*\backslash C_2\times \Qq^*_{|\cdot |< 1 }\ri by\in \Qq^*_{|\cdot |< 1 }\ri (ax,by)\in P$,
        \item $(a,b)\in C_2\times \Qq^*_{|\cdot |\leq 1 }, (x,y)\in \Qq^*\backslash C_2\times \Qq^*_{|\cdot |< 1 }\ri (ax,by)\in \Qq^*\backslash C_2\times \Qq^*_{|\cdot |< 1 } \ri (ax,by)\in P $.
    \end{itemize}
\end{itemize}

\end{Example}
\begin{Example}[Internal category that is neither a groupoid nor an internal \textsf{S}-category]
    By slightly modifying the above example, we can get an internal category that is neither a groupoid nor an internal \textsf{S}-category.
     \begin{center}
        \begin{tikzcd}
    |[alias=HL]|\Rr^*&|[alias=HM]|\Rr^*\times \Rr^* &|[alias=HR]|\Rr^*\\
    &&\\
    |[alias=LL]| C_2 &|[alias=LM]| C_2\times \Qq^*_{|\cdot |\leq 1 }\cup  {\Qq^*_{|\cdot |< 1 }}^{\times 2}&|[alias=LR]|\Qq^*_{|\cdot |\leq 1 }\\ 
    \arrow[>->, from=LL, to=HL]
    \arrow[>->, from=HL, to=HM, "i_0"]
    \arrow[>->, from=LL, to=LM, "\overline{i_0}"]   
    \arrow[>->, from=LR, to=HR] 
    \arrow[>->, from=LM, to=HM]
    \arrow[->>, from=LM, to=LR, shift left=4, "\overline{p_1}"]
    \arrow[->>, from=LM, to=LR, shift right=5, "\overline{p_1}"']
    \arrow[->>, from=HM, to=HR, shift left=4, "p_1"]
    \arrow[->>, from=HM, to=HR, shift right=4, "p_1"']
    \arrow[>->, from=LR, to=LM, shift left=0, "\overline{i_1}"description]
    \arrow[>->, from=HR, to=HM, shift left=0, "i_1"description]
        \end{tikzcd}
    \end{center}
    
Now the morphism $\sigma : (a,b)\mapsto (a\mm, b)$ is not monotone as for example $\sigma(1/2, 1/2)=(2, 1/2)\not\in  C_2\times \Qq^*_{|\cdot |\leq 1 }\cup  {\Qq^*_{|\cdot |< 1 }}^{\times 2}$. It is not an internal \textsf{S}-category because $(1/2, 1/2)-(1,1/2)=(1/2, 1)\not\in  C_2\times \Qq^*_{|\cdot |\leq 1 }\cup  {\Qq^*_{|\cdot |< 1 }}^{\times 2}$. 
\end{Example}

Notice that for all the previous examples, all groups are abelian and therefore all examples are valid in both categories $\crpogrp$ and $\cpogrp$.

\section{Acknowledgements}
The author thanks his supervisor M. Gran for some suggestions on the subject of the article and for providing the example (\ref{EXG}), as well as A. Montoli and the anonymous referee for the discussions and corrections on the first version of this article.

\section*{Declarations}

\begin{itemize}
\item This research was funded by Fonds de la Recherche Scientifique-FRS. 
\item The author declares no conflicts of interest.
\item Ethics declaration: not applicable.
\end{itemize}

\end{document}